\documentclass{m2an}
\usepackage{float}
\usepackage{amsmath,amsfonts,amssymb,amsthm}
\usepackage{enumerate}
\usepackage[utf8]{inputenc}
\usepackage[T1]{fontenc}
\usepackage[english]{babel}
\usepackage{graphicx}
\usepackage{color}
\usepackage{pgf,tikz,pgfplots}
\usepackage{lmodern}
\usepackage[colorlinks=true]{hyperref}
\hypersetup{urlcolor=blue, citecolor=blue}

\addcontentsline{toc}{section}{Acknowledgement}

\newtheorem{deff}{Definition}[section]
\newtheorem{prop}[deff]{Proposition}
\newtheorem{thm}[deff]{Theorem}
\newtheorem{lem}[deff]{Lemma}
\newtheorem{cor}[deff]{Corollary}
\newtheorem{rmk}[deff]{\it Remark}

\newcounter{cst}

\def\1{{\bf 1}}
\def\be{\begin{equation}}
\def\ee{\end{equation}}

\begin{document}

\title{On some new results on anisotropic singular perturbations of second-order elliptic operators}
\author{David Maltese}\address{LAMA, Univ. Gustave Eiffel, Univ. Paris Est Cr\'eteil, CNRS, F-77454 Marne-la-Vall\'ee, France.\\ \href{david.maltese@univ-eiffel.fr}{\tt david.maltese@univ-eiffel.fr}}
\author{Chokri Ogabi}\address{LAMA, Univ. Gustave Eiffel, Univ. Paris Est Cr\'eteil, CNRS, F-77454 Marne-la-Vall\'ee, France.\\ \href{chokri.ogabi@univ-eiffel.fr}{\tt chokri.ogabi@univ-eiffel.fr}}
\begin{abstract} 
In this article, we deal with  some problems involving a class of singularly perturbed elliptic operators. We prove the asymptotic preserving of a general Galerkin method associated to a semilinear problem. We use a particular Galerkin approximation to estimate the convergence rate on the whole domain, for the linear problem. Finally, we study the asymptotic behavior of the semigroup generated.
\end{abstract}
\maketitle

\section{\protect\bigskip introduction}

Anisotropic singular perturbations problems were introduced by Chipot in \cite{Chipot}, these problems can model diffusion phenomena when the diffusion parameters become small in certain directions. We refer the reader to \cite{chip-guesm1},\cite{chip-guesm2},\cite{chip-guesm3},\cite{chip-guesm4},\cite{guesm1},\cite{guesm2},\cite{guesm3},\cite{ogabi1},\cite{ogabi2},\cite{ogabi3} for several works on this
topic. In this article, we will study some new theoretical aspects which
have not been studied before for these problems. 

Let us consider the following perturbed elliptic problem 
\begin{equation}
\beta (u_{\epsilon })-\text{div}(A_{\epsilon }\nabla u_{\epsilon })=f~\text{%
in}~\Omega ,  \label{u:probcont}
\end{equation}%
supplemented with the boundary condition 
\begin{equation}
u_{\epsilon }=0~\text{on}~\partial \Omega .  \label{u:probbound}
\end{equation}%
Here, $\Omega =\omega _{1}\times \omega _{2}$ where $\omega _{1}$ and $%
\omega _{2}$ are two bounded open sets of $%
\mathbb{R}
^{q}$ and $%
\mathbb{R}
^{N-q},$ with $N>q\geq 1$, and $f\in L^{2}(\Omega ).$ We denote by $%
x=(x_{1},...,x_{N})=(X_{1},X_{2})\in 
\mathbb{R}
^{q}\times 
\mathbb{R}
^{N-q}$ i.e. we split the coordinates into two parts. With this notation we
set%
\begin{equation*}
\nabla =(\partial _{x_{1}},...,\partial _{x_{N}})^{T}=\binom{\nabla _{X_{1}}%
}{\nabla _{X_{2}}},\text{ }
\end{equation*}%
where 
\begin{equation*}
\nabla _{X_{1}}=(\partial _{x_{1}},...,\partial _{x_{q}})^{T}\text{ and }%
\nabla _{X_{2}}=(\partial _{x_{q+1}},...,\partial _{x_{N}})^{T}.
\end{equation*}%
The function $A=(a_{ij})_{1\leq i,j\leq N}:\Omega \rightarrow \mathcal{M}%
_{N}(\mathbb{R})$ satisfies the ellipticity assumptions

\begin{itemize}
\item There exists $\lambda > 0$ such that for a.e. $x \in \Omega$ 
\begin{equation}
A\xi \cdot \xi \geq \lambda \left\vert \xi \right\vert ^{2}~\text{for any}~
\xi \in 
\mathbb{R}
^{N}.  \label{hypA1}
\end{equation}

\item The coefficients of $A$ are bounded, that is 
\begin{equation}  \label{hypA2}
a_{ij}\in L^{\infty }(\Omega )~\text{for any}~ (i,j) \in \{1,2,....,N\}^2.
\end{equation}
\end{itemize}
We have decomposed $A$ into four blocks

\begin{equation*}
A=%
\begin{pmatrix}
A_{11} & A_{12} \\ 
A_{21} & A_{22}%
\end{pmatrix}%
,
\end{equation*}
where $A_{11}$, $A_{22}$ are $q\times q$ and $(N-q)\times (N-q)$ matrices
respectively. For $\epsilon \in (0,1]$ we have set 
\begin{equation*}
A_{\epsilon }=%
\begin{pmatrix}
\epsilon ^{2}A_{11} & \epsilon A_{12} \\ 
\epsilon A_{21} & A_{22}%
\end{pmatrix}.%
\end{equation*}%
The function $\beta :%
\mathbb{R}
\rightarrow 
\mathbb{R}
$ satisfies the following conditions:%
\begin{equation}
\beta \text{ is continuous and nondecreasing with }\beta (0)=0.
\label{hypBeta1}
\end{equation}%
\begin{equation}
\exists M\geq 0:\forall s\in 
\mathbb{R}
,\left\vert \beta (s)\right\vert \leq M\left( 1+\left\vert s\right\vert
\right).  \label{hypBeta2}
\end{equation}
The weak formulation of problem ($\ref{u:probcont}$)-($\ref{u:probbound}$%
) is 
\begin{equation}
\left\{ 
\begin{array}{l}
\int_{\Omega }\beta (u_{\epsilon })\varphi dx+\int_{\Omega }A_{\epsilon
}\nabla u_{\epsilon }\cdot \nabla \varphi dx=\int_{\Omega }f\text{ }\varphi
dx\text{, }\forall \varphi \in H_{0}^{1}(\Omega )\text{}
\\ 
u_{\epsilon }\in H_{0}^{1}(\Omega )\text{,\ \ \ \ \ \ \ \ \ \ \ \ \ \ \ \ \ \
\ \ \ \ \ \ \ \ \ \ \ \ \ \ \ \ \ \ \ \ \ \ \ \ \ \ \ \ \ \ \ \ \ \ \ \ \ \
\ \ \ \ \ \ \ \ \ \ \ \ \ \ \ \ \ \ \ \ \ \ \ \ \ \ \ }%
\end{array}%
\right.   \label{probcontweak}
\end{equation}%
where the existence and the uniqueness follow from assumptions $(\ref%
{hypA1}-\ref{hypBeta2})$.
The limit problem is given by 
\begin{equation}
\beta (u)-\text{div}_{X_{2}}(A_{22}\nabla u)=f~\text{on}~\Omega ,
\label{u:probcontlimit}
\end{equation}%
supplemented with the boundary condition 
\begin{equation}
u(X_{1},\cdot )=0~\text{in}~\partial \omega _{2},~\text{for}~X_{1}\in \omega
_{1}.  \label{u:problimitbound}
\end{equation}%
We introduce the functional space 
\begin{equation*}
H_{0}^{1}(\Omega ;\omega _{2})=\left\{ v\in L^{2}(\Omega )~\text{such that}%
~\nabla _{X_{2}}v\in L^{2}(\Omega )^{N-q}\text{ and for a.e. }~X_{1}\in
\omega _{1},v(X_{1},\cdot )\in H_{0}^{1}(\omega _{2})\right\},
\end{equation*}%
equipped with the norm $\left\Vert \nabla _{X_{2}}(\cdot )\right\Vert
_{L^{2}(\Omega )^{N-q}}$. Notice that this norm is equivalent to 
\begin{equation*}
\left( \left\Vert (\cdot )\right\Vert _{L^{2}(\Omega )}^{2}+\left\Vert
\nabla _{X_{2}}(\cdot )\right\Vert _{L^{2}(\Omega )^{N-q}}\right) ^{1/2},
\end{equation*}%
thanks to Poincar\'{e}'s inequality 
\begin{equation}
\left\Vert v\right\Vert _{L^{2}(\Omega )}\leq C_{\omega _{2}}\left\Vert
\nabla _{X_{2}}v\right\Vert _{L^{2}(\Omega )^{N-q}},~\text{for any}~v\in
H_{0}^{1}(\Omega ;\omega _{2}).  \label{sob}
\end{equation}%
One can prove that $H_{0}^{1}(\Omega ;\omega _{2})$ is a Hilbert space. The
space $H_{0}^{1}(\Omega )$ will be normed by $\left\Vert \nabla (\cdot
)\right\Vert _{L^{2}(\Omega )^{N}}$. \ One can check immediately that the
embedding $H_{0}^{1}(\Omega )\hookrightarrow H_{0}^{1}(\Omega ,\omega _{2})$
is continuous.

The weak formulation of the limit problem $(\ref{u:probcontlimit})-(\ref{u:problimitbound})$ is given by 
\begin{equation}
\left\{ 
\begin{array}{ll}
\begin{array}{c}
\int_{\omega _{2}}\beta (u)\mathbf{(}X_{1},\cdot )\psi dX_{2}+\int_{\omega
_{2}}A_{22}\mathbf{(}X_{1},\cdot )\nabla _{X_{2}}u\mathbf{(}X_{1},\cdot
)\cdot \nabla _{X_{2}}\psi dX_{2} \\ 
=\int_{\omega _{2}}f\mathbf{(}X_{1},\cdot )\text{ }\psi dX_{2}\text{, }%
\forall \psi \in H_{0}^{1}(\omega _{2})\text{\ \ \ \ }%
\end{array}
&  \\ 
u\mathbf{(}X_{1},\cdot )\in H_{0}^{1}(\omega _{2})\text{,\ for a.e. }%
X_{1}\in \omega _{1}\text{ \ \ \ \ \ \ \ \ \ \ \ \ \ \ \ \ \ \ \ \ \ \ \ \ \
\ \ \ \ \ \ \ \ \ \ \ \ \ \ \ \ \ \ \ \ \ \ \ \ \ \ \ \ \ \ \ \ \ \ \ \ } & 
\text{ }%
\end{array}%
\right.  \label{probcontlimit}
\end{equation}%
This problem has been studied in \cite{ogabi1}, and the author proved the
following (see Proposition 4 in the above reference)

\begin{thm}
\label{thrmconvergence}Under assumptions $(%
\ref{hypA1}),(\ref{hypA2}),(\ref{hypBeta1})$ and $(\ref{hypBeta2})$ we have: 
\begin{equation*}
u_{\epsilon }\rightarrow u\text{ in }L^{2}(\Omega ),\text{ }\epsilon
\nabla _{X_{1}}u_{\epsilon }\rightarrow 0\ \text{ in }L^{2}(\Omega )^{q}%
\text{ and }\nabla _{X_{2}}u_{\epsilon }\rightarrow \nabla _{X_{2}}u\ \text{
in }L^{2}(\Omega )^{N-q},
\end{equation*}
where $u_{\epsilon }$ is the unique solution to $(\ref%
{probcontweak})$ in $H_{0}^{1}(\Omega )$ and $u$ is the unique
solution to $(\ref{probcontlimit})$ in $H_{0}^{1}(\Omega ;\omega _{2}).$
\end{thm}
Remark that for $\varphi \in H_{0}^{1}(\Omega ;\omega _{2})$, and for a.e $%
X_{1}$ in $\omega _{1}$ we have $\varphi (X_{1},\cdot )\in H_{0}^{1}(\omega
_{2})$. By testing with $\varphi (X_{1},\cdot )$ in (\ref{probcontlimit}) and by integrating over $%
\omega _{1}$ we get
\begin{equation}
\int_{\Omega }\beta (u)\varphi dx+\int_{\Omega }A_{22}\nabla _{X_{2}}u\cdot
\nabla _{X_{2}}\varphi dx=\int_{\Omega }f\text{ }\varphi dx,\text{\ }\forall
\varphi \in H_{0}^{1}(\Omega ;\omega _{2}).  \label{probcontlimitomega}
\end{equation}
This paper is organized as follows:
\begin{itemize}
\item As a first main result, we will prove the asymptotic preserving of the general Galerkin method for the elliptic problem (\ref{u:probcont})-(\ref{u:probbound}). This concept has been introduced by S. Jin in \cite{Sin} and it could be illustrated by
the following commutative diagram%
\begin{equation*}
\begin{array}{c}
\text{ }P_{\epsilon ,n}\underrightarrow{\text{ \ \ \ \ }^{{\small %
n\rightarrow \infty }\text{ \ }}}\text{\ \ \ \ }P_{\epsilon } \\ 
\left. ^{\substack{  \\ {\small \epsilon \rightarrow 0}}}\right\downarrow 
\text{ \ \ \ \ \ \ \ \ \ \ \ \ \ \ \ \ \ }\left\downarrow ^{\substack{  \\ %
\epsilon \rightarrow 0}}\right. \text{\ } \\ 
P_{n}\text{ }\underrightarrow{\text{ \ \ \ }^{{\small n\rightarrow \infty }%
\text{ \ \ \ }}}\text{\ \ \ }P_{0}%
\end{array}%
,
\end{equation*}%
here, $P_{\epsilon ,n}$ is the Galerkin approximation of the infinite dimensional perturbed problem $P_{\epsilon }$, and $P_{n}$ is the Galerkin approximation of the infinite dimensional limit problem $P_{0}$. We will derive an estimation of the error for a general Galerkin method, and by using a C\'{e}a type lemmas we prove the asymptotic-preserving of the method.

\item As a second main result, we will prove, in the linear case, a new
result on the estimation of the global convergence rate, such a result is of
the form $\left\Vert \nabla _{X_{2}}(u_{\epsilon }-u)\right\Vert
_{L^{2}(\Omega )^{N-q}}\leq C\epsilon .$ This estimation is an improvement
of the local one proved by Chipot and Guesmia in \cite{chip-guesm2}. Our
arguments are based on the use of a particular Galerkin approximation
constructed by a tensor product.

\item In section 4, we will prove our third main result on the asymptotic behavior of the semigroup generated by the perturbed elliptic operator $\text{div}(A_{\epsilon }\nabla
\cdot ),$ and we will give a simple application to linear parabolic problems.
\end{itemize}

Finally, to make the paper readable, we put some intermediate technical
lemmas in the appendix.

\section{Main theorems for the elliptic problem}

\begin{deff}
\label{def-approximanteHilbert}Let $(V_{n})$ be a sequence of finite
dimensional subspaces of a Hilbert space $H$. We say that $(V_{n})$
approximates $H$, if for every $w\in H.$ 
\begin{equation*}
\underset{v\in V_{n}}{\inf }\left\Vert w-v\right\Vert _{H}\longrightarrow 0%
\text{ as }n\rightarrow \infty.
\end{equation*}
\end{deff}

For a sequence $(V_{n})$ of a finite dimensional spaces of $H_{0}^{1}(\Omega
)$, and for every $\epsilon \in (0,1]$ and $n\in 
\mathbb{N}
$, we denote $u_{\epsilon ,n\text{ }}$ the unique solution of%
\begin{equation}
\left\{ 
\begin{array}{l}
\int_{\Omega }\beta (u_{\epsilon ,n})\varphi dx+\int_{\Omega }A_{\epsilon
}\nabla u_{\epsilon ,n}\cdot \nabla \varphi dx=\int_{\Omega }f\text{ }%
\varphi dx\text{, }\forall \varphi \in V_{n}\text{. } \\ 
u_{\epsilon ,n}\in V_{n}\text{.}%
\end{array}%
\right.   \label{Galerkin0}
\end{equation}%
We suppose that 
\begin{equation}
\partial _{x_{i}}a_{ij}\in L^{\infty }(\Omega ),\partial _{x_{j}}a_{ij}\in
L^{\infty }(\Omega )\text{ for }i=1,...,q\text{ and }j=q+1,...,N.
\label{hypAd1}
\end{equation}%
We have the following:

\begin{thm}
\label{mainthm0}Let $\Omega =\omega _{1}\times \omega _{2}$ where $\omega
_{1}$ and $\omega _{2}$ are two bounded open sets of $%
\mathbb{R}
^{q}$ and $%
\mathbb{R}
^{N-q}$ respectively, with $N>q\geq 1$. Suppose that $f\in L^{2}(\Omega )$ and assume $(%
\ref{hypA1}),(\ref{hypA2}),(\ref{hypBeta1})$, $(\ref{hypBeta2})$, and $(\ref{hypAd1})$. Let $%
(V_{n})$ be a sequence of finite dimensional spaces of $H_{0}^{1}(\Omega )$
which approximates it in the sense of Definition \ref%
{def-approximanteHilbert}. Let $(u_{\epsilon ,n})$ be the sequence
of solutions of $(\ref{Galerkin0})$ then we have: 
\begin{equation*}
\underset{_{\epsilon }}{\lim }(\underset{n}{\lim }u_{\epsilon ,n})=\underset{%
n}{\lim }(\underset{_{\epsilon }}{\lim }u_{\epsilon ,n})=u,\text{ in }%
H_{0}^{1}(\Omega ;\omega _{2}),
\end{equation*}%
where $u$ is the unique solution of $(\ref{probcontlimit})$ in $%
H_{0}^{1}(\Omega ;\omega _{2}).$
\end{thm}

Our second result concerns the estimation of the rate of convergence for problem (\ref{probcontweak}) in the linear case, this result
could be seen as a refinement of the following result proved in \cite{chip-guesm2}
:%
\begin{equation}
\forall\omega _{1}^{\prime }\subset \subset \text{ }\omega
_{1} \text{ open}:\left\Vert \nabla _{X_{2}}(u_{\epsilon }-u)\right\Vert _{L^{2}(\omega
_{1}^{\prime }\times \omega _{2})}=O(\epsilon ),\text{ and }\left\Vert
\nabla _{X_{1}}(u_{\epsilon }-u)\right\Vert _{L^{2}(\omega _{1}^{\prime
}\times \omega _{2})}=O(1).  \label{interiorestim}
\end{equation}%
In the above reference, the authors have supposed that 
\begin{equation}
\nabla _{X_{1}}f\in L^{2}(\Omega )^{q},  \label{hypFad1}
\end{equation}%
assumption (\ref{hypAd1}), and that $\nabla _{X_{1}}A_{22}\in L^{\infty }(\Omega )$.
Our contribution consists in extending (\ref{interiorestim}) to the whole
domain $\Omega $, to obtain such a result we take some additional hypothesis
on $A$ and $f$, namely:%
\begin{equation}
\text{For a.e. }X_{2}\in \omega _{2}:f(\cdot ,X_{2})\in H_{0}^{1}(\omega
_{1}),  \label{hypFad2}
\end{equation}%
and 
\begin{equation}
\text{ The block }A_{22}\text{ depends only on }X_{2}.  \label{hypAd2}
\end{equation}

\begin{thm}
\label{mainthm} Let $\Omega =\omega _{1}\times \omega _{2}$ where $\omega
_{1}$ and $\omega _{2}$ are two bounded open sets of $%
\mathbb{R}
^{q}$ and $%
\mathbb{R}
^{N-q}$ respectively, with $N>q\geq 1.$ Suppose that $\beta =0,$ and let us assume that $%
A $ satisfies $(\ref{hypA1})$, $(\ref{hypA2})$, $(\ref{hypAd1})$ and $(\ref%
{hypAd2})$. Let $f\in L^{2}(\Omega )$ such that $(\ref{hypFad1})$ and $(\ref%
{hypFad2}),$ then there exists $C>0$ depending on $f$, $\lambda $, $C_{\omega
_{2}}$ and $A$ such that: 
\begin{equation*}
\forall \epsilon \in (0,1]:\left\Vert \nabla _{X_{2}}(u_{\epsilon
}-u)\right\Vert _{L^{2}(\Omega )^{N-q}}\leq C\epsilon ,
\end{equation*}%
where $u_{\epsilon }$ is the unique solution of ($\ref{probcontweak}$) in $%
H_{0}^{1}(\Omega )$ and $u$ is the unique solution to $(\ref{probcontlimit})$
in $H_{0}^{1}(\Omega ;\omega _{2}).$ Moreover, we have: 
\begin{equation*}
u\in H_{0}^{1}(\Omega )\text{ and }\nabla _{X_{1}}(u_{\epsilon
}-u)\rightharpoonup 0\text{ weakly in }L^{2}(\Omega )^{q}, \text{as } \epsilon\rightarrow 0.
\end{equation*}
The constant $C$ is of the form $C_{1}\left\Vert \nabla _{X_{1}}f\right\Vert
_{L^{2}(\Omega )^{q}}+C_{2}\left\Vert f\right\Vert _{L^{2}(\Omega )}$ where $%
C_{1},C_{2}$ are dependent on $A,\lambda ,C_{\omega _{2}}.$
\end{thm}
The proof of this theorem will be done in two steps. First, we give the
proof in the case $f\in H_{0}^{1}(\omega _{1})\otimes H_{0}^{1}(\omega _{2})$%
, and next that we conclude by a density argument. Let us recall this density rule, which will be used throughout this article: If $(E,\tau )$ and 
$(F,\tau ^{\prime })$ are two topological spaces such that $E\subset F$, and 
$E$ is dense in $F$ and the canonical injection $E\rightarrow F$ is
continuous, then every dense subset in $(E,\tau )$ is dense in $(F,\tau
^{\prime }).$

\begin{rmk}
The hypothesis $(\ref{hypFad2})$ is necessary to obtain the global
boundedness of $\nabla _{X_{1}}(u_{\epsilon }-u).$ We can observe that
through this 2d example, we take 
\begin{equation*}
A=id_{2}\text{, }f:(x_{1},x_{2})\longmapsto \cos (x_{1})\sin (x_{2}),\text{
and }\Omega =(0,\pi )\times (0,\pi ).
\end{equation*}%
In this case, we have $u(x_{1},x_{2})=\cos (x_{1})\sin (x_{2}).$ The
quantity $\left\Vert \nabla _{X_{1}}(u_{\epsilon }-u)\right\Vert
_{L^{2}(\Omega )^{q}}$ could not be bounded. Indeed, if we suppose the
converse then according to Theorem \ref{thrmconvergence} there exists a subsequence still labeled $(u_{\epsilon })$ such
that $\nabla _{X_{1}}(u_{\epsilon }-u)\rightharpoonup 0$ weakly in $%
L^{2}(\Omega )^{q},$ and $\left\Vert \nabla _{X_{2}}(u_{\epsilon }-u)\right\Vert _{L^{2}(\Omega
)^{N-q}}\rightarrow 0.$ Whence $u\in H_{0}^{1}(\Omega )$ which is a
contradiction.
\end{rmk}

Let us finish by giving this remark which will be used later in section 4.

\begin{rmk}
\label{rem-semigroup-importante}Suppose that $\beta :s\longmapsto \mu s$,
for some $\mu >0$, and suppose that assumptions of Theorem \ref{mainthm}
hold, then we have the same results of Theorem \ref{mainthm} with the same constants. Assume, in
addition, that the block $A_{12}$ satisfies the following: 
\begin{equation}
\partial _{x_{i}x_{j}}^{2}a_{ij}\in L^{2}(\Omega )\text{, for }i=1,...,q%
\text{, }j=q+1,...,N,  \label{hypA12-seconde}
\end{equation}%
then we have:%
\begin{equation*}
\forall\epsilon\in(0,1]:\left\Vert \nabla_{X_{2}}(u_{\epsilon }-u)\right\Vert _{L^{2}(\Omega )^{N-q}}\leq \frac{\epsilon }{%
\mu }\left( C_{1}^{\prime }\left\Vert \nabla _{X_{1}}f\right\Vert
_{L^{2}(\Omega )^{q}}+C_{2}^{\prime }\left\Vert f\right\Vert _{L^{2}(\Omega
)}\right),
\end{equation*}%
where $C_{1}^{\prime },C_{2}^{\prime }$ are only dependent on $A,\lambda
,C_{\omega _{2}}.$
\end{rmk}

\section{ The Analysis of a general Galerkin method}

\subsection{Preliminaries}

Let $V\subset H_{0}^{1}(\Omega )$ be a closed subspace of $H_{0}^{1}(\Omega
,\omega _{2}).$ Notice that $V$ is closed in $H_{0}^{1}(\Omega )$, thanks to
the continuous embedding $H_{0}^{1}(\Omega )\hookrightarrow H_{0}^{1}(\Omega
,\omega _{2}).$ Let $f\in L^{2}(\Omega )$, we denote by $u_{\epsilon ,V,f}$
the unique solution of 
\begin{equation}
\left\{ 
\begin{array}{l}
\int_{\Omega }\beta (u_{\epsilon ,V,f})\varphi dx+\int_{\Omega }A_{\epsilon
}\nabla u_{\epsilon ,V,f}\cdot \nabla \varphi dx=\int_{\Omega }f\text{ }%
\varphi dx\text{, }\forall \varphi \in V\text{\ \ \ \ \ } \\ 
u_{\epsilon ,V,f}\in V\text{.}%
\end{array}%
\right.  \label{Galerkin}
\end{equation}%
We denote by $u_{V,f}$ the unique solution of 
\begin{equation}
\left\{ 
\begin{array}{l}
\int_{\Omega }\beta (u_{V,f})\varphi dx+\int_{\Omega }A_{22}\nabla
_{X_{2}}u_{V,f}\cdot \nabla _{X_{2}}\varphi dx=\int_{\Omega }f\text{ }%
\varphi dx\text{, }\forall \varphi \in V\text{\ \ \ } \\ 
u_{V,f}\in V\text{.}%
\end{array}%
\right.  \label{Galerkin-limit}
\end{equation}%
Under assumptions of Theorem \ref{thrmconvergence}, one can prove by using the Schauder fixed point theorem that $u_{\epsilon ,V,f}$ exists. For the existence of $u_{V,f}$ see Appendix \ref{AppendixC}. The
uniqueness, for the two problems, follows immediately from (\ref{hypA1}) and
(\ref{hypBeta1})$.$ Now, let us begin by some preliminary lemmas

\begin{lem}
\label{estbasic} Under assumptions of Theorem \ref{thrmconvergence} and for
any $\epsilon \in (0,1]$, we have the following bounds: 
\begin{equation}
\left\Vert \nabla u_{\epsilon ,V,f}\right\Vert _{L^{2}(\Omega
)^{N}}\leq \frac{C_{\Omega}\left\Vert f\right\Vert _{L^{2}(\Omega )}}{%
\lambda \epsilon^{2} }, \text{and } \left\Vert \nabla u_{\epsilon ,f}\right\Vert _{L^{2}(\Omega
)^{N}}\leq \frac{C_{\Omega}\left\Vert f\right\Vert _{L^{2}(\Omega )}}{%
\lambda \epsilon^{2} }.
 \label{bound1}
\end{equation}%
\begin{equation}
\left\Vert \nabla _{X_{2}}u_{V,f}\right\Vert _{L^{2}(\Omega )^{N-q}}\leq 
\frac{C_{\omega _{2}}\left\Vert f\right\Vert _{L^{2}(\Omega )}}{\lambda }, \text{and } \left\Vert \nabla _{X_{2}}u_{f}\right\Vert _{L^{2}(\Omega )^{N-q}}\leq \frac{%
C_{\omega _{2}}\left\Vert f\right\Vert _{L^{2}(\Omega )}}{\lambda }.
\label{bound2}
\end{equation}%
\begin{equation}
\left\Vert \beta (u_{\epsilon ,V,f})\right\Vert _{L^{2}(\Omega )}\leq
\frac{M}{\epsilon^2}\left( \left\vert \Omega \right\vert ^{\frac{1}{2}}+\frac{C_{\Omega
}^{2}\left\Vert f\right\Vert _{L^{2}(\Omega )}}{\lambda }\right), \text{and } \left\Vert \beta (u_{\epsilon,f})\right\Vert _{L^{2}(\Omega )}\leq
\frac{M}{\epsilon^2}\left( \left\vert \Omega \right\vert ^{\frac{1}{2}}+\frac{C_{\Omega
}^{2}\left\Vert f\right\Vert _{L^{2}(\Omega )}}{\lambda }\right).
\label{bound3}
\end{equation}%
\begin{equation}
\left\Vert \beta (u_{V,f})\right\Vert _{L^{2}(\Omega )}\leq M\left(
\left\vert \Omega \right\vert ^{\frac{1}{2}}+\frac{C_{\omega
_{2}}^{2}\left\Vert f\right\Vert _{L^{2}(\Omega )}}{\lambda }\right), \text{and } \left\Vert \beta (u_{f})\right\Vert _{L^{2}(\Omega )}\leq M\left(
\left\vert \Omega \right\vert ^{\frac{1}{2}}+\frac{C_{\omega
_{2}}^{2}\left\Vert f\right\Vert _{L^{2}(\Omega )}}{\lambda }\right).
\label{bound4}
\end{equation}%
Here, $C_{\Omega }$ is the Poincar\'{e} constant of $\Omega$, and $u_{\epsilon ,f}$, $u_{f}$ are the unique solutions of $(\ref{probcontweak})$ and  $(\ref{probcontlimit})$ respectively. 
\end{lem}

\begin{proof}
These bounds follow easily from a suitable choice of the test functions,
monotonicity and ellipticity assumptions. Let us prove, for example, the second inequality in (\ref%
{bound2}) and the second inequality in (\ref{bound4}).
According to Theorem \ref{thrmconvergence} one can take $\varphi =u_{f}$ in (%
\ref{probcontlimitomega}), using ellipticity assumption and the fact that $%
\int_{\Omega }\beta (u_{f})u_{f}dx\geq 0$ (thanks to (\ref{hypBeta1})) we
obtain%
\begin{equation*}
\lambda \int_{\Omega }\left\vert \nabla _{X_{2}}u_{f}\right\vert ^{2}dx\leq
\int_{\Omega }f\text{ }u_{f}dx.
\end{equation*}%
By the Cauchy-Schwarz inequality and Poincar\'{e}'s inequality (\ref{sob}), we obtain the second inequality of (\ref{bound2}).
Now, by using assumption (\ref{hypBeta2}), we obtain%
\begin{equation*}
\left\vert \beta (u_{f})\right\vert ^{2}\leq M^{2}\left( 1+\left\vert
u_{f}\right\vert \right) ^{2}.
\end{equation*}%
Integrating over $\Omega $ and applying Minkowski inequality, (\ref{sob}),
and (\ref{bound2}) we obtain the second inequality of (\ref{bound4}).
\end{proof}

By using the above lemma, one can prove the following C\'{e}a type lemma

\begin{lem}
\label{Cea-nonlinear}Under assumptions of Theorem \ref{thrmconvergence} we have:
\begin{equation}
\Vert \nabla _{X_{2}}(u_{V,f}-u_{f})\Vert _{L^{2}(\Omega )^{N-q}}\leq C_{c%
\acute{e}a}\left( \inf_{v\in V}\Vert \nabla _{X_{2}}(v-u_{f})\Vert
_{L^{2}(\Omega )^{N-q}}\right) ^{\frac{1}{2}},  \label{Cea1}
\end{equation}%
and for any $\epsilon \in (0,1]$:
\begin{equation}
\Vert \nabla (u_{\epsilon ,V,f}-u_{\epsilon ,f})\Vert _{L^{2}(\Omega
)^{N}}\leq \frac{C_{c\acute{e}a}^{\prime }}{\epsilon ^{2}}\left( \inf_{v\in
V}\Vert \nabla (v-u_{\epsilon,f})\Vert _{L^{2}(\Omega )^{N}}\right) ^{\frac{1}{2}%
},  \label{Cea2}
\end{equation}%
where 
\begin{equation*}
C_{c\acute{e}a}^{2}=\frac{1}{\lambda}\left[ 2M C_{\omega _{2}}\left( \left\vert \Omega \right\vert ^{\frac{1}{2}}+%
\frac{C_{\omega _{2}}^{2}\left\Vert f\right\Vert _{L^{2}(\Omega )}}{\lambda }%
\right) +\Vert A_{22}\Vert _{L^{\infty }(\Omega )}\frac{%
2C_{\omega _{2}}\left\Vert f\right\Vert _{L^{2}(\Omega )}}{\lambda }\right],
\end{equation*}%
and%
\begin{equation*}
C_{c\acute{e}a}^{{\prime }^{2}}=\frac{1}{\lambda}\left[ 2MC_{\Omega }\left( \left\vert \Omega \right\vert ^{%
\frac{1}{2}}+\frac{C_{\Omega }^{2}\left\Vert f\right\Vert _{L^{2}(\Omega )}}{%
\lambda }\right) +\Vert A\Vert _{L^{\infty }(\Omega )}\frac{%
2C_{\Omega }\left\Vert f\right\Vert _{L^{2}(\Omega )}}{\lambda }\right].
\end{equation*}%
\end{lem}

\begin{proof}
The proofs of these two inequalities are similar. So, let us prove the first
one. Using the Galerkin orthogonality one has, for $v\in V:$ 
\begin{multline*}
\int_{\Omega }\left( \beta (u_{V,f})-\beta (u_{f}\right)
)(u_{V,f}-u_{f})dx+\lambda\Vert \nabla _{X_{2}}(u_{V,f}-u_{f})\Vert _{L^{2}(\Omega
)^{N-q}}^{2} \\
\leq\int_{\Omega }\left( \beta (u_{V,f})-\beta (u_{f}\right)
)(v-u_{f})dx+\int_{\Omega }A_{22}\nabla _{X_{2}}(u_{V,f}-u_{f})\cdot \nabla
_{X_{2}}(v-u_{f})dx.
\end{multline*}%
Using the fact that $\int_{\Omega }\left( \beta (u_{V,f})-\beta
(u_{f}\right) )(u_{V,f}-u_{f})dx\geq 0$, then by Cauchy-Schwarz and Poincar%
\'{e}'s inequalities we derive%
\begin{multline*}
\lambda\Vert \nabla _{X_{2}}(u_{V,f}-u_{f})\Vert _{L^{2}(\Omega )^{N-q}}^{2}\leq 
\left[ C_{\omega _{2}}\left\Vert \beta (u_{V,f})-\beta (u_{f})\right\Vert _{L^{2}(\Omega
)}+\Vert A_{22}\Vert _{L^{\infty }(\Omega )}\Vert \nabla
_{X_{2}}(u_{V,f}-u_{f})\Vert _{L^{2}(\Omega )^{N-q}}\right] \\
\times \Vert \nabla _{X_{2}}(v-u_{f})\Vert _{L^{2}(\Omega )^{N-q}}.
\end{multline*}%
Now, by using (\ref{bound2}), (\ref{bound4}) and the triangle inequality we
obtain%
\begin{multline*}
\lambda\Vert \nabla _{X_{2}}(u_{V,f}-u_{f})\Vert _{L^{2}(\Omega )^{N-q}}^{2}\leq \\
2\left[ M C_{\omega _{2}}\left( \left\vert \Omega \right\vert ^{\frac{1}{2}}+\frac{C_{\omega_{2}}^{2}\left\Vert f\right\Vert _{L^{2}(\Omega )}}{\lambda}\right)
+\Vert A_{22}\Vert _{L^{\infty }(\Omega )}\frac{C_{\omega
_{2}}\left\Vert f\right\Vert _{L^{2}(\Omega )}}{\lambda }\right] \times
\Vert \nabla _{X_{2}}(v-u_{f})\Vert _{L^{2}(\Omega )^{N-q}},
\end{multline*}
and (\ref{Cea1}) follows.
\end{proof}

\begin{rmk}
\label{Cea-linear} $1)$ If $\beta =0$ (the linear case), then we have for any $\epsilon \in (0,1]:$  
\begin{equation*}
\Vert \nabla u_{\epsilon ,V,f}-\nabla u_{\epsilon ,f}\Vert _{L^{2}(\Omega
)^{N}}\leq \frac{\Vert A\Vert _{L^{\infty }(\Omega )}}{\lambda \epsilon ^{2}}%
\inf_{v\in V}\Vert \nabla v-\nabla u_{\epsilon ,f}\Vert _{L^{2}(\Omega
)^{N}}.
\end{equation*}%
\begin{equation*}
\Vert \nabla _{X_{2}}u_{V,f}-\nabla _{X_{2}}u_{f}\Vert _{L^{2}(\Omega
)^{N-q}}\leq \frac{\Vert A_{22}\Vert _{L^{\infty }(\Omega )}}{\lambda }%
\inf_{v\in V}\Vert \nabla _{X_{2}}v-\nabla _{X_{2}}u_{f}\Vert _{L^{2}(\Omega
)^{N-q}}.
\end{equation*}

$2)$ If $\beta $ is Lipschitz, then we can obtain estimations similar to
those of the linear case.
\end{rmk}

\subsection{Error estimates in the Galerkin method}

\begin{lem}
\label{lem:errorest} Under assumptions of Theorem \ref{thrmconvergence},
suppose in addition that $(\ref{hypAd1})$ holds, then we have for every $\epsilon\in(0,1]$: 
\begin{equation*}
\left\Vert \nabla _{X_{2}}(u_{\epsilon ,V,f}-u_{V,f})\right\Vert
_{L^{2}(\Omega )^{N-q}}\leq \epsilon \left( C_{1}\left\Vert \nabla
_{X_{1}}u_{V,f}\right\Vert _{L^{2}(\Omega )^{q}}+C_{2}\left\Vert
f\right\Vert _{L^{2}(\Omega )}\right), 
\end{equation*}%
and 
\begin{equation*}
\left\Vert \nabla _{X_{1}}(u_{\epsilon ,V,f}-u_{V,f})\right\Vert
_{L^{2}(\Omega )^{q}}\leq \frac{1}{\sqrt{2}}\left( C_{1}\left\Vert \nabla
_{X_{1}}u_{V,f}\right\Vert _{L^{2}(\Omega )^{q}}+C_{2}\left\Vert
f\right\Vert _{L^{2}(\Omega )}\right), 
\end{equation*}%
where 
\begin{equation*}
C_{1}=\left( \frac{4(C+C^{\prime })}{\lambda }\right) ^{\frac{1}{2}}\text{
and }C_{2}=\frac{2\sqrt{C^{\prime \prime }}C_{\omega _{2}}}{\lambda ^{3/2}}.
\end{equation*}%
Here, $C,C^{\prime },\text{and } C^{\prime \prime }$ are given by $(\ref{const1}),(%
\ref{const2})$ and $(\ref{const3}).$ Notice that these constants are
independent of $\epsilon ,V$ and $f.$
\end{lem}

\begin{proof}
By subtracting (\ref{Galerkin-limit}) from (\ref{Galerkin}) we get, for
every $v\in V:$ 
\begin{multline*}
\int_{\Omega }(\beta (u_{\epsilon ,V,f})-\beta (u_{V,f}))vdx+\epsilon
^{2}\int_{\Omega }A_{11}\nabla _{X_{1}}u_{\epsilon ,V,f}\cdot \nabla
_{X_{1}}vdx \\
+\epsilon \int_{\Omega }A_{12}\nabla _{X_{2}}u_{\epsilon ,V,f}\cdot \nabla
_{X_{1}}vdx+\epsilon \int_{\Omega }A_{21}\nabla _{X_{1}}u_{\epsilon
,V,f}\cdot \nabla _{X_{2}}vdx \\
+\int_{\Omega }A_{22}\nabla _{X_{2}}(u_{\epsilon ,V,f}-u_{V,f})\cdot \nabla
_{X_{2}}vdx=0\text{,}
\end{multline*}
Testing with $v=u_{\epsilon ,V,f}-u_{V,f}$, we obtain 
\begin{multline*}
\int_{\Omega }(\beta (u_{\epsilon ,V,f})-\beta (u_{V,f}))(u_{\epsilon ,V,f}-u_{V,f})dx+\int_{\Omega }A_{\epsilon }\nabla (u_{\epsilon
,V,f}-u_{V,f})\cdot \nabla (u_{\epsilon ,V,f}-u_{V,f}) \\
=-\epsilon ^{2}\int_{\Omega }A_{11}\nabla _{X_{1}}u_{V,f}\cdot \nabla
_{X_{1}}(u_{\epsilon ,V,f}-u_{V,f})dx-\epsilon \int_{\Omega }A_{12}\nabla
_{X_{2}}u_{V,f}\cdot \nabla _{X_{1}}(u_{\epsilon ,V,f}-u_{V,f})dx \\
-\epsilon \int_{\Omega }A_{21}\nabla _{X_{1}}u_{V,f}\cdot \nabla
_{X_{2}}(u_{\epsilon ,V,f}-u_{V,f})dx.
\end{multline*}%
whence, by using (\ref{hypBeta1}) and the ellipticity assumption we get 
\begin{multline*}
\epsilon ^{2}\lambda \int_{\Omega }\left\vert \nabla _{X_{1}}(u_{\epsilon
,V,f}-u_{V,f})\right\vert ^{2}dx+\lambda \int_{\Omega }\left\vert \nabla
_{X_{2}}(u_{\epsilon ,V,f}-u_{V,f})\right\vert ^{2}dx\leq \\
-\epsilon ^{2}\int_{\Omega }A_{11}\nabla _{X_{1}}u_{V,f}\cdot \nabla
_{X_{1}}(u_{\epsilon ,V,f}-u_{V,f})dx-\epsilon \int_{\Omega }A_{12}\nabla
_{X_{2}}u_{V,f}\cdot \nabla _{X_{1}}(u_{\epsilon ,V,f}-u_{V,f})dx \\
-\epsilon \int_{\Omega }A_{21}\nabla _{X_{1}}u_{V,f}\cdot \nabla
_{X_{2}}(u_{\epsilon ,V,f}-u_{V,f})dx.
\end{multline*}
Let us estimate the first and the last term of the second member in the above
inequality. By using Young's inequality we obtain%
\begin{multline*}
-\epsilon ^{2}\int_{\Omega }A_{11}\nabla _{X_{1}}u_{V,f}\cdot \nabla
_{X_{1}}(u_{\epsilon ,V,f}-u_{V,f})dx \\
\leq \frac{\epsilon ^{2}\lambda }{2}\int_{\Omega }\left\vert \nabla
_{X_{1}}(u_{\epsilon ,V,f}-u_{V,f})\right\vert ^{2}dx+\epsilon ^{2}\frac{%
\left\Vert A_{11}\right\Vert _{L^{\infty }(\Omega )}^{2}}{2\lambda }%
\int_{\Omega }\left\vert \nabla _{X_{1}}u_{V,f}\right\vert ^{2}dx,
\end{multline*}%
and 
\begin{multline*}
-\epsilon \int_{\Omega }A_{21}\nabla _{X_{1}}u_{V,f}\cdot \nabla
_{X_{2}}(u_{\epsilon ,V,f}-u_{V,f})dx \\
\leq \epsilon ^{2}\frac{\left\Vert A_{21}\right\Vert _{L^{\infty }(\Omega
)}^{2}}{2\lambda }\int_{\Omega }\left\vert \nabla _{X_{1}}u_{V,f}\right\vert
^{2}dx+\frac{\lambda }{2}\int_{\Omega }\left\vert \nabla
_{X_{2}}(u_{\epsilon ,V,f}-u_{V,f})\right\vert ^{2}dx,
\end{multline*}%
thus 
\begin{multline}\label{11}
\frac{\epsilon ^{2}\lambda }{2}\left\Vert \nabla _{X_{1}}(u_{\epsilon
,V,f}-u_{V,f})\right\Vert _{L^{2}(\Omega )}^{2}+\frac{\lambda }{2}\left\Vert
\nabla _{X_{2}}(u_{\epsilon ,V,f}-u_{V,f})\right\Vert _{L^{2}(\Omega
)^{N-q}}^{2}   \\
\leq C\epsilon ^{2}\int_{\Omega }\left\vert \nabla
_{X_{1}}u_{V,f}\right\vert ^{2}dx-\epsilon \int_{\Omega }A_{12}\nabla
_{X_{2}}u_{V,f}\cdot \nabla _{X_{1}}(u_{\epsilon ,V,f}-u_{V,f})dx,  \notag
\end{multline}
\begin{eqnarray}\label{11}
\end{eqnarray}
where 
\begin{equation}
C=\frac{\left\Vert A_{21}\right\Vert _{L^{\infty }(\Omega )}^{2}+\left\Vert
A_{11}\right\Vert _{L^{\infty }(\Omega )}^{2}}{2\lambda }.  \label{const1}
\end{equation}%
Now, we estimate $-\epsilon \int_{\Omega }A_{12}\nabla _{X_{2}}u_{V,f}\cdot
\nabla _{X_{1}}(u_{\epsilon ,V,f}-u_{V,f})dx.$ Since $u_{\epsilon
,V,f}-u_{V,f}\in H_{0}^{1}(\Omega )$ and $\partial _{x_{i}}a_{ij}\in
L^{\infty }(\Omega )$, $\partial _{x_{j}}a_{ij}\in L^{\infty }(\Omega )$ for 
$i=1,...,q$ and $j=q+1,...,N,$ (assumption (\ref{hypAd1})) then we can show by a simple density argument
that for $i=1,...,q$ and $j=q+1,...,N$, $\ \partial _{x_{k}}(a_{ij}$ $%
(u_{\epsilon ,V,f}-u_{V,f}))\in L^{2}(\Omega )$ and:%
\begin{equation*}
\partial _{x_{k}}(a_{ij}(u_{\epsilon ,V,f}-u_{V,f}))=(u_{\epsilon
,V,f}-u_{V,f})\partial _{x_{k}}a_{ij}+a_{ij}\partial _{x_{k}}(u_{\epsilon
,V,f}-u_{V,f}), \text{ for }k=i,j).
\end{equation*}%
Whence 
\begin{eqnarray*}
-\epsilon \int_{\Omega }A_{12}\nabla _{X_{2}}u_{V,f}\cdot \nabla
_{X_{1}}(u_{\epsilon ,V,f}-u_{V,f})dx &=&-\epsilon
\sum_{i=1}^{q}\sum_{j=q+1}^{N}\int_{\Omega }a_{ij}\partial
_{x_{j}}u_{V,f}\partial _{x_{i}}(u_{\epsilon ,V,f}-u_{V,f})dx \\
&=&-\epsilon \sum_{i=1}^{q}\sum_{j=q+1}^{N}\int_{\Omega }\partial
_{x_{i}}(a_{ij}(u_{\epsilon ,V,f}-u_{V,f}))\partial _{x_{j}}u_{V,f}dx \\
&&+\epsilon \sum_{i=1}^{q}\sum_{j=q+1}^{N}\int_{\Omega }(u_{\epsilon
,V,f}-u_{V,f})\partial _{x_{i}}a_{ij}\partial _{x_{j}}u_{V,f}dx \\
&=&-\epsilon \sum_{i=1}^{q}\sum_{j=q+1}^{N}\int_{\Omega }\partial
_{x_{j}}(a_{ij}(u_{\epsilon ,V,f}-u_{V,f}))\partial _{x_{i}}u_{V,f}dx \\
&&+\epsilon \sum_{i=1}^{q}\sum_{j=q+1}^{N}\int_{\Omega }(u_{\epsilon
,V,f}-u_{V,f})\partial _{x_{i}}a_{ij}\partial _{x_{j}}u_{V,f}dx,
\end{eqnarray*}%
where we have used $\int_{\Omega }\partial _{x_{i}}(a_{ij}(u_{\epsilon
,V,f}-u_{V,f}))\partial _{x_{j}}u_{V,f}dx=\int_{\Omega }\partial
_{x_{j}}(a_{ij}(u_{\epsilon ,V,f}-u_{V,f}))\partial _{x_{i}}u_{V,f}dx$ which
follows by a simple density argument (recall that $u_{V,f}\in H_{0}^{1}(\Omega
)$). Therefore%
\begin{eqnarray}
-\epsilon \int_{\Omega }A_{12}\nabla _{X_{2}}u_{V,f}\cdot \nabla
_{X_{1}}(u_{\epsilon ,V,f}-u_{V,f})dx &=&-\epsilon
\sum_{i=1}^{q}\sum_{j=q+1}^{N}\int_{\Omega }(u_{\epsilon
,V,f}-u_{V,f})\partial _{x_{j}}a_{ij}\partial _{x_{i}}u_{V,f}dx
\label{semigroup-annex} \\
&&-\epsilon \sum_{i=1}^{q}\sum_{j=q+1}^{N}\int_{\Omega }a_{ij}\partial
_{x_{j}}(u_{\epsilon ,V,f}-u_{V,f})\partial _{x_{i}}u_{V,f}dx  \notag \\
&&+\epsilon \sum_{i=1}^{q}\sum_{j=q+1}^{N}\int_{\Omega }(u_{\epsilon
,V,f}-u_{V,f})\partial _{x_{i}}a_{ij}\partial _{x_{j}}u_{V,f}dx.  \notag
\end{eqnarray}
By Young's and Poincar\'{e}'s inequalities we obtain 
\begin{multline*}
-\epsilon \int_{\Omega }A_{12}\nabla _{X_{2}}u_{V,f}\cdot \nabla
_{X_{1}}(u_{\epsilon ,V,f}-u_{V,f})dx\leq \frac{\lambda }{4}\int_{\Omega
}\left\vert \nabla _{X_{2}}(u_{\epsilon ,V,f}-u_{V,f})\right\vert ^{2}dx \\
+C^{\prime }\epsilon ^{2}\int_{\Omega }\left\vert \nabla
_{X_{1}}u_{V,f}\right\vert ^{2}dx+C^{\prime \prime }\epsilon
^{2}\int_{\Omega }\left\vert \nabla _{X_{2}}u_{V,f}\right\vert ^{2}dx,
\end{multline*}%
where 
\begin{equation}
C^{\prime }=\frac{3\left[ C_{\omega _{2}}\underset{1\leq i\leq q,q+1\leq
j\leq N}{\max }\left\Vert \partial _{x_{j}}a_{ij}\right\Vert _{L^{\infty
}(\Omega )}(N-q)\right] ^{2}+3\left( \underset{1\leq i\leq q,q+1\leq j\leq N}%
{\max }\left\Vert a_{ij}\right\Vert _{L^{\infty }(\Omega )}(N-q)\right) ^{2}%
}{\lambda }.  \label{const2}
\end{equation}%
and%
\begin{equation}
C^{\prime \prime }=\frac{3\left[ qC_{\omega _{2}}\underset{1\leq i\leq
q,q+1\leq j\leq N}{\max }\left\Vert \partial _{x_{i}}a_{ij}\right\Vert
_{L^{\infty }(\Omega )}\right] ^{2}}{\lambda }.  \label{const3}
\end{equation}%
By using (\ref{bound2}) we obtain
\begin{multline}\label{12}
-\epsilon \int_{\Omega }A_{12}\nabla _{X_{2}}u_{V,f}\cdot \nabla
_{X_{1}}(u_{\epsilon ,V,f}-u_{V,f})dx\leq   \\
\frac{\lambda }{4}\int_{\Omega }\left\vert \nabla _{X_{2}}(u_{\epsilon
,V,f}-u_{V,f})\right\vert ^{2}dx+C^{\prime }\epsilon ^{2}\int_{\Omega
}\left\vert \nabla _{X_{1}}u_{V,f}\right\vert ^{2}dx+\epsilon ^{2}C^{\prime
\prime }\left( \frac{C_{\omega _{2}}\left\Vert f\right\Vert _{L^{2}(\Omega )}%
}{\lambda }\right) ^{2}.  \notag
\end{multline}%
\begin{eqnarray}\label{12}
\end{eqnarray}
Combining (\ref{11}) and (\ref{12}) we get 
\begin{multline*}
\frac{\epsilon ^{2}\lambda }{2}\left\Vert \nabla _{X_{1}}(u_{\epsilon
,V,f}-u_{V,f})\right\Vert _{L^{2}(\Omega )^{q}}^{2}+\frac{\lambda }{4}%
\left\Vert \nabla _{X_{2}}(u_{\epsilon ,V,f}-u_{V,f})\right\Vert
_{L^{2}(\Omega )^{N-q}}^{2} \\
\leq \epsilon ^{2}\left( (C+C^{\prime })\int_{\Omega }\left\vert \nabla
_{X_{1}}u_{V,f}\right\vert ^{2}dx+C^{\prime \prime }\left( \frac{C_{\omega
_{2}}\left\Vert f\right\Vert _{L^{2}(\Omega )}}{\lambda }\right) ^{2}\right)
,
\end{multline*}%
and the proof is finished.
\end{proof}

Using the triangle inequality, the above Lemma and (\ref{Cea1}) we
obtain the following estimation of the global error between $u_{\epsilon
,V,f}$ and $u_{f}$.

\begin{cor}
\label{corollaire-errorglobal} Under assumptions of Lemma $\ref%
{lem:errorest}$ we have for any $\epsilon \in (0,1]:$%
\begin{equation*}
\left\Vert \nabla _{X_{2}}(u_{\epsilon ,V,f}-u_{f})\right\Vert
_{L^{2}(\Omega )^{N-q}}\leq \epsilon \left( C_{1}\left\Vert \nabla
_{X_{1}}u_{V,f}\right\Vert _{L^{2}(\Omega )^{q}}+C_{2}\left\Vert
f\right\Vert _{L^{2}(\Omega )}\right) +C_{c\acute{e}a}\left( \inf_{v\in
V}\Vert \nabla _{X_{2}}(v-u_{f})\Vert _{L^{2}(\Omega )^{N-q}}\right) ^{\frac{%
1}{2}}.
\end{equation*}
\end{cor}

Now, we give an important remark which will be used to prove the inequality given in Remark \ref{rem-semigroup-importante} .

\begin{rmk}
\label{rem-pour semigroup1}When $\beta (s)=\mu s$ for some $\mu >0$ and when
the block $A_{12}$ satisfies assumption $(\ref{hypA12-seconde})$, then by performing some integration by parts in the last term of $(\ref%
{semigroup-annex})$, and by using the fact that \[
\left\Vert u_{V,f}\right\Vert _{L^{2}(\Omega )}\leq \frac{1}{\mu }\left\Vert
f\right\Vert _{L^{2}(\Omega )},
\]we can obtain the following estimation:%
\begin{equation*}
\forall \epsilon \in (0,1]:\left\Vert \nabla _{X_{2}}(u_{\epsilon ,V,f}-u_{V,f})\right\Vert
_{L^{2}(\Omega )^{N-q}}\leq \epsilon \left( C_{1}^{\prime }\left\Vert \nabla
_{X_{1}}u_{V,f}\right\Vert _{L^{2}(\Omega )^{q}}+\frac{C_{2}^{\prime }}{\mu }%
\left\Vert f\right\Vert _{L^{2}(\Omega )}\right),
\end{equation*}%
where $C_{1}^{\prime },$ $C_{2}^{\prime }>0$ are independent of $f,$ $V,$ $%
\mu $ and $\epsilon $.
\end{rmk}

\subsection{Proof of Theorem \protect\ref{mainthm0}}

Let $(V_{n})$ be a sequence of finite dimensional subspaces which approximates $%
H_{0}^{1}(\Omega )$ in the sense of Definition \ref{def-approximanteHilbert}%
. Using the density of $H_{0}^{1}(\Omega )$ in $H_{0}^{1}(\Omega ,\omega
_{2})$ (Lemma \ref{lem:density}, Appendix \ref{appendixA}), one can
check easily that $(V_{n})$ approximates $H_{0}^{1}(\Omega ,\omega _{2})$ in
the same sense. Therefore, one has:%
\begin{equation}
\text{For every }\epsilon \in (0,1]:\inf_{v\in V_{n}}\Vert \nabla
(v-u_{\epsilon ,f})\Vert _{L^{2}(\Omega )^{N}}\rightarrow 0\text{ as }%
n\rightarrow \infty ,  \label{denst1}
\end{equation}%
and 
\begin{equation}
\inf_{v\in V_{n}}\Vert \nabla _{X_{2}}(v-u_{f})\Vert _{L^{2}(\Omega
)^{N-q}}\rightarrow 0\text{ as }n\rightarrow \infty.  \label{denst2}
\end{equation}%
According to Lemma \ref{lem:errorest}, (\ref{Cea1}) and ($\ref{Cea2}$) we
have, for every $n\in 
\mathbb{N}
$ and $\epsilon \in (0,1]:$%
\begin{equation}
\left\Vert \nabla _{X_{2}}(u_{\epsilon ,V_{n},f}-u_{V_{n},f})\right\Vert
_{L^{2}(\Omega )^{N-q}}\leq \epsilon \left( C_{1}\left\Vert \nabla
_{X_{1}}u_{V_{n},f}\right\Vert _{L^{2}(\Omega )^{q}}+C_{2}\left\Vert
f\right\Vert _{L^{2}(\Omega )}\right),  \label{errorest-bis}
\end{equation}%
\begin{equation}
\left\Vert \nabla _{X_{2}}(u_{V_{n},f}-u_{f})\right\Vert _{L^{2}(\Omega
)^{N-q}}\leq C_{c\acute{e}a}\left( \inf_{v\in V_{n}}\Vert \nabla
_{X_{2}}(v-u_{f})\Vert _{L^{2}(\Omega )^{N-q}}\right) ^{\frac{1}{2}},
\label{cea1bis}
\end{equation}%
and%
\begin{equation}
\Vert \nabla (u_{\epsilon ,V_{n},f}-u_{\epsilon ,f})\Vert _{L^{2}(\Omega
)^{N}}\leq \frac{C_{c\acute{e}a}^{\prime }}{\epsilon ^{2}}\left( \inf_{v\in
V_{n}}\Vert \nabla (v-u_{\epsilon ,f})\Vert _{L^{2}(\Omega )^{N}}\right) ^{%
\frac{1}{2}}.  \label{cea2bis}
\end{equation}
\begin{itemize}
\item Fix $\epsilon $ and pass to the limit in ($\ref{cea2bis}$) by using (\ref%
{denst1}), we get%
\begin{equation*}
u_{\epsilon ,V_{n},f}\rightarrow u_{\epsilon ,f}\text{ as }n\rightarrow
\infty \text{ in }H_{0}^{1}(\Omega ),
\end{equation*}%
in particular, by using the continuous embedding $H_{0}^{1}(\Omega )\hookrightarrow
H_{0}^{1}(\Omega ,\omega _{2})$ we deduce 
\begin{equation*}
u_{\epsilon ,V_{n},f}\rightarrow u_{\epsilon ,f}\text{ as }n\rightarrow
\infty \text{ in }H_{0}^{1}(\Omega ,\omega _{2}).
\end{equation*}%
Now, passing to the limit as $\epsilon \rightarrow 0$ by using Theorem \ref%
{thrmconvergence}, we get 
\begin{equation}
\underset{_{\epsilon }}{\lim }(\underset{n}{\lim }u_{\epsilon
,V_{n},f})=u_{f}\text{ in }H_{0}^{1}(\Omega ,\omega _{2}).
\label{asympreserve1}
\end{equation}

\item Fix $n$ and passe to the limit as $\epsilon \rightarrow 0$ in (\ref%
{errorest-bis}), we get 
\begin{equation*}
u_{\epsilon ,V_{n},f}\rightarrow u_{V_{n},f}\text{ as }\epsilon \rightarrow 0%
\text{ in }H_{0}^{1}(\Omega ,\omega _{2}).
\end{equation*}%
Now, passing to the limit as $n\rightarrow \infty $ in (\ref{cea1bis}) by
using ($\ref{denst2}$), we get 
\begin{equation}
\underset{_{n}}{\lim }(\underset{\epsilon }{\lim }u_{\epsilon
,V_{n},f})=u_{f}\text{ in }H_{0}^{1}(\Omega ,\omega _{2}).
\label{asympreserve2}
\end{equation}
\end{itemize}

Finally, Theorem \ref{mainthm0} follows from (\ref{asympreserve1}) and (\ref%
{asympreserve2}).

\subsection{\protect\bigskip Proof of Theorem \protect\ref{mainthm}}

Throughout this subsection, we will suppose that $\beta =0.$ The key of the
proof of Theorem \ref{mainthm} is based on the control of the quantity $%
\left\Vert \nabla _{X_{1}}u_{V,f}\right\Vert _{L^{2}(\Omega )^{q}}$
independently of $V.$ In fact, we need the following:

\begin{lem}
\label{estgradu} Let us assume that $A$ satisfies $(\ref{hypA1})$, $(\ref%
{hypA2})$, and that $A_{22}$ satisfies $(\ref{hypAd2})$.
Let $V_{1}$ and $V_{2}$ be two finite dimensional subspaces of $%
H_{0}^{1}(\omega _{1})$ and $H_{0}^{1}(\omega _{2})$ respectively. Let $f\in
V_{1}\otimes V_{2},$ and let $u_{V,f}$ be the unique solution in $%
V=V_{1}\otimes V_{2}$ to: 
\begin{equation}
\int_{\Omega }A_{22}(X_{2})\nabla _{X_{2}}u_{V,f}\cdot \nabla
_{X_{2}}vdx=\int_{\Omega }fvdx\text{, \ \ }\forall v\in V_{1}\otimes V_{2},
\label{galerkin-tensor}
\end{equation}%
then we have: 
\begin{equation*}
\left\Vert \nabla _{X_{1}}u_{V,f}\right\Vert _{L^{2}(\Omega )^{q}}\leq
C_{3}\left\Vert \nabla _{X_{1}}f\right\Vert _{L^{2}(\Omega )^{q}},
\end{equation*}%
where $C_{3}$ is given by $C_{3}=\frac{\sqrt{q}C_{\omega _{2}}}{\lambda }$.
\end{lem}

\begin{proof}
The proof is based on the difference quotient method (see for instance \cite{trudinger} page 168). Let $v=$ $\varphi \otimes \psi \in V_{1}\otimes V_{2}$.
The function $X_{1}\longmapsto \int_{\omega _{2}}A_{22}(X_{2})\nabla
_{X_{2}}u_{V,f}(X_{1},X_{2})\cdot \nabla _{X_{2}}\psi dX_{2}$ belongs to $%
V_{1}$. In fact $u_{V,f}=\sum\limits_{finite}\varphi _{i}\otimes \psi _{i}$, and whence $\int_{\omega _{2}}A_{22}(X_{2})\nabla _{X_{2}}u_{V,f}\cdot
\nabla _{X_{2}}\psi dX_{2}$ is a linear combination of $\varphi _{i}$, thanks to
the linearity of the integral. Similarly, the function $X_{1}\longmapsto
\int_{\omega _{2}}f(X_{1},X_{2})\psi dX_{2}$ belongs to $V_{1}.$ Now, testing with $v$ in (%
\ref{galerkin-tensor}), we derive: 
\begin{equation*}
\int_{\omega _{1}}\left( \int_{\omega _{2}}\left\{ A_{22}(X_{2})\nabla
_{X_{2}}u_{V,f}\cdot \nabla _{X_{2}}\psi -f\psi \right\} dX_{2}\right)
\varphi dX_{1}=0,
\end{equation*}%
thus, when $\varphi $ run through a set of an orthogonal basis of the
euclidean space $V_{1}$ equipped with the $L^{2}(\omega _{1})-$scalar
product, one can deduce that for a.e. $X_{1}\in \omega _{1}$ :%
\begin{equation*}
\int_{\omega _{2}}A_{22}(X_{2})\nabla _{X_{2}}u_{V,f}(X_{1},X_{2})\cdot
\nabla _{X_{2}}\psi dX_{2}=\int_{\omega _{2}}f(X_{1},X_{2})\psi dX_{2},%
\text{ }\forall \psi \in V_{2}.
\end{equation*}%
Now, fix $i\in \{1,...,q\}$. Let $\omega _{1}^{\prime }\subset \subset
\omega _{1}$ open, for any $0<h<d(\omega _{1}^{\prime },\partial \omega
_{1}) $ and for any $(X_{1},X_{2})\in \omega _{1}^{\prime }\times \omega
_{2} $ we denote $\tau
_{h}u_{V,f}(x)=u_{V,f}(x_{1},...x_{i}+h,...,x_{q},X_{2})$. According to the
above equality, we get for a.e. $X_{1}\in $ $\omega _{1}^{\prime }$ and for
every $\psi \in V_{2}:$ 
\begin{equation*}
\int_{\omega _{2}}A_{22}(X_{2})\nabla _{X_{2}}\left\{ \tau
_{h}u_{V,f}(X_{1},X_{2})-u_{V,f}(X_{1},X_{2})\right\} \nabla _{X_{2}}\psi
dX_{2}=\int_{\omega _{2}}\left\{ \tau _{h}f(X_{1},X_{2})-f(X_{1},X_{2})\right\}
\psi dX_{2}.
\end{equation*}%
For every $w\in $ $V_{1}\otimes V_{2}$, and for every $X_{1}$ fixed the
function $w(X_{1},\cdot )$ belongs to $V_{2}$, so one can take $\psi =\tau
_{h}u_{V,f}(X_{1},\cdot )-u_{V,f}(X_{1},\cdot )$ as a test function in the
above equality. Therefore, by using the Cauchy-Schwarz inequality, the
ellipticity assumption, and Poincar\'{e}'s inequality (\ref{sob}), we obtain:%
\begin{equation*}
\int_{\omega _{2}}\left\vert \tau _{h}u_{V,f}(X_{1},\cdot)-u_{V,f}(X_{1},\cdot)\right\vert ^{2}dX_{2}\leq \frac{C_{\omega _{2}}^{4}}{\lambda ^{2}}%
\int_{\omega _{2}}\left\vert \tau _{h}f(X_{1},\cdot )-f(X_{1},\cdot
)\right\vert ^{2}dX_{2}.
\end{equation*}
Now, integrating the above inequality over $\omega _{1}^{\prime }$, yields%
\begin{equation*}
\int_{\omega _{1}^{\prime }\times \omega _{2}}\left\vert \tau
_{h}u_{V,f}-u_{V,f}\right\vert ^{2}dx\leq \frac{C_{\omega _{2}}^{4}}{\lambda
^{2}}\int_{\omega _{1}^{\prime }\times \omega _{2}}\left\vert \tau
_{h}f-f\right\vert ^{2}dx.
\end{equation*}%
Since $\nabla _{X_{1}}f\in L^{2}(\Omega )^{q}$, then 
\begin{equation*}
\int_{\omega _{1}^{\prime }\times \omega _{2}}\left\vert \tau
_{h}f-f\right\vert ^{2}dx\leq \left\Vert \nabla _{X_{1}}f\right\Vert
_{L^{2}(\Omega )^{q}}^{2}h^{2}.
\end{equation*}%
Finally, we obtain 
\begin{equation*}
\int_{\omega _{1}^{\prime }\times \omega _{2}}\left\vert \frac{\tau
_{h}u_{V,f}-u_{V,f}}{h}\right\vert ^{2}dx\leq \frac{C_{\omega
_{2}}^{4}\left\Vert \nabla _{X_{1}}f\right\Vert _{L^{2}(\Omega )^{q}}^{2}}{%
\lambda ^{2}}.
\end{equation*}
Therefore
\begin{equation*}
\left\Vert D_{x_{i}}u_{V,f}\right\Vert _{L^{2}(\Omega )}\leq \frac{C_{\omega
_{2}}^{2}}{\lambda }\left\Vert \nabla _{X_{1}}f\right\Vert _{L^{2}(\Omega )^{q}},
\end{equation*}
and hence 
\begin{equation*}
\left\Vert \nabla _{X_{1}}u_{V,f}\right\Vert _{L^{2}(\Omega )^{q}}\leq
C_{3}\left\Vert \nabla _{X_{1}}f\right\Vert _{L^{2}(\Omega )^{q}},
\end{equation*}%
with $C_{3}=\frac{\sqrt{q}C_{\omega _{2}}^{2}}{\lambda }$.
\end{proof}

\begin{rmk}
\label{rem-poursemigroup2}We have a similar result when $(\ref%
{galerkin-tensor})$ is replaced by 
\begin{equation*}
\mu \int_{\Omega }u_{V,f}vdx+\int_{\Omega }A_{22}(X_{2})\nabla
_{X_{2}}u_{V,f}\cdot \nabla _{X_{2}}vdx=\int_{\Omega }fvdx\text{, \ \ }%
\forall v\in V_{1}\otimes V_{2},
\end{equation*}%
where $\mu >0$. In this case, we obtain the following: 
\begin{equation*}
\left\Vert \nabla _{X_{1}}u_{V,f}\right\Vert _{L^{2}(\Omega )^{q}}\leq \frac{%
\sqrt{q}}{\mu }\left\Vert \nabla _{X_{1}}f\right\Vert _{L^{2}(\Omega )^{q}}.
\end{equation*}
\end{rmk}

Now, we can refine the estimations of Lemma \ref{lem:errorest} as follows

\begin{lem}
\label{errorest} Under assumptions of Lemmas \ref{lem:errorest} and \ref{estgradu} we have: 
\begin{multline*}
\left\Vert \nabla _{X_{2}}u_{\epsilon ,V,f}-\nabla _{X_{2}}u_{f}\right\Vert
_{L^{2}(\Omega )^{N-q}}\leq \epsilon \left( C_{1}C_{3}\left\Vert \nabla
_{X_{1}}f\right\Vert _{L^{2}(\Omega )^{q}}+C_{2}\left\Vert f\right\Vert
_{L^{2}(\Omega )}\right) \\
+\frac{\Vert A_{22}\Vert _{L^{\infty }(\Omega )}}{\lambda }\inf_{v\in
V_{1}\otimes V_{2}}\left\Vert \nabla _{X_{2}}v-\nabla
_{X_{2}}u_{f}\right\Vert _{L^{2}(\Omega )^{N-q}},
\end{multline*}%
and 
\begin{equation*}
\left\Vert \nabla _{X_{1}}u_{\epsilon ,V,f}\right\Vert _{L^{2}(\Omega
)^{q}}\leq \frac{1}{\sqrt{2}}\left( C_{1}C_{3}\left\Vert \nabla
_{X_{1}}f\right\Vert _{L^{2}(\Omega )^{q}}+C_{2}\left\Vert f\right\Vert
_{L^{2}(\Omega )}\right) +C_{3}\left\Vert \nabla _{X_{1}}f\right\Vert
_{L^{2}(\Omega )^{q}}.
\end{equation*}
\end{lem}

\begin{proof}
We have 
\begin{equation*}
\left\Vert \nabla _{X_{2}}u_{\epsilon ,V,f}-\nabla _{X_{2}}u_{f}\right\Vert
_{L^{2}(\Omega )^{N-q}}\leq \left\Vert \nabla _{X_{2}}u_{\epsilon
,V,f}-\nabla _{X_{2}}u_{V,f}\right\Vert _{L^{2}(\Omega )^{N-q}} 
+\left\Vert \nabla _{X_{2}}u_{V,f}-\nabla _{X_{2}}u_{f}\right\Vert
_{L^{2}(\Omega )^{N-q}}.
\end{equation*}%
By using Lemma \ref{lem:errorest} and Lemma \ref{estgradu} we obtain that 
\begin{equation*}
\left\Vert \nabla _{X_{2}}u_{\epsilon ,V,f}-\nabla
_{X_{2}}u_{V,f}\right\Vert _{L^{2}(\Omega )^{N-q}}\leq \epsilon \left(
C_{1}C_{3}\left\Vert \nabla _{X_{1}}f\right\Vert _{L^{2}(\Omega
)^{q}}+C_{2}\left\Vert f\right\Vert _{L^{2}(\Omega )}\right),
\end{equation*}%
and by using Remark \ref{Cea-linear}, we deduce 
\begin{equation*}
\left\Vert \nabla _{X_{2}}u_{V,f}-\nabla _{X_{2}}u_{f}\right\Vert
_{L^{2}(\Omega )^{N-q}}\leq \frac{\Vert A_{22}\Vert _{L^{\infty }(\Omega )}}{%
\lambda }\inf_{v\in V_{1}\otimes V_{2}}\Vert \nabla _{X_{2}}v-\nabla _{X_{2}}u_{f}\Vert
_{L^{2}(\Omega )^{N-q}}.
\end{equation*}%
By using the above inequalities, we get the expected result.
The second inequality follows from the triangle inequality and 
Lemmas \ref{lem:errorest} and \ref{estgradu}.
\end{proof}

\begin{rmk}
\label{rem-poursemigroup3}Let $\beta (s)=\mu s,$ for a certain $%
\mu >0$. Under assumptions of the above Lemma and (\ref{hypA12-seconde}) we obtain, by combining Remarks \ref{rem-pour semigroup1} and \ref%
{rem-poursemigroup2}, the estimation:%
\begin{equation*}
\forall\epsilon\in(0,1]:\left\Vert \nabla _{X_{2}}(u_{\epsilon ,V,f}-u_{V,f})\right\Vert _{L^{2}(\Omega )}\leq \frac{%
\epsilon }{\mu }\left( \sqrt{q}C_{1}^{\prime }\left\Vert \nabla
_{X_{1}}f\right\Vert _{L^{2}(\Omega )^{q}}+C_{2}^{\prime }\left\Vert
f\right\Vert _{L^{2}(\Omega )}\right).
\end{equation*}
\end{rmk}

Now, we are able to give the first convergence result

\begin{lem}
\label{lem:error} Suppose that assumptions of Lemmas \ref{lem:errorest} and \ref{estgradu} hold. Let $f\in H_{0}^{1}(\omega _{1})\otimes H_{0}^{1}(\omega
_{2})$, then we have for any $\epsilon\in(0,1]$: 
\begin{equation*}
\left\Vert \nabla _{X_{2}}u_{\epsilon ,f}-\nabla _{X_{2}}u_{f}\right\Vert
_{L^{2}(\Omega )^{N-q}}\leq \epsilon \left( C_{1}C_{3}\left\Vert \nabla
_{X_{1}}f\right\Vert _{L^{2}(\Omega )^{q}}+C_{2}\left\Vert f\right\Vert
_{L^{2}(\Omega )}\right),
\end{equation*}%
and 
\begin{equation*}
\left\Vert \nabla _{X_{1}}u_{\epsilon ,f}\right\Vert _{L^{2}(\Omega
)^{q}}\leq \frac{1}{\sqrt{2}}\left( C_{1}C_{3}\left\Vert \nabla
_{X_{1}}f\right\Vert _{L^{2}(\Omega )^{q}}+C_{2}\left\Vert f\right\Vert
_{L^{2}(\Omega )}\right) +C_{3}\left\Vert \nabla _{X_{1}}f\right\Vert
_{L^{2}(\Omega )^{q}}.
\end{equation*}
\end{lem}

\begin{proof}
Let $(V_{n}^{(1)})_{n\geq 0}$ and $(V_{n}^{(2)})_{n\geq 0}$ be two
nondecreasing sequences of finite dimensional subspaces of $H_{0}^{1}(\omega
_{1})$ and $H_{0}^{1}(\omega _{2})$ respectively, such that $\cup{V_{n}^{(1)}}$ and $\cup{V_{n}^{(2)}}$ are dense in $H_{0}^{1}(\omega
_{1})$ and $H_{0}^{1}(\omega _{2})$ respectively, and such that $f\in
V_{0}^{(1)}\otimes V_{0}^{(2)},$ such sequences always exit. Indeed, let $%
\left\{ e_{i}^{(1)}\right\} _{i\in 
\mathbb{N}
}$ and $\left\{ e_{i}^{(2)}\right\} _{i\in 
\mathbb{N}
}$ be Hilbert bases of $H_{0}^{1}(\omega _{1})$ and $H_{0}^{1}(\omega
_{2}) $ respectively, then $\cup _{n\geq
0}span(e_{0}^{(1)},...,e_{n}^{(1)})$ and $\cup _{n\geq
0}span(e_{0}^{(2)},...,e_{n}^{(2)})$ are dense in $H_{0}^{1}(\omega _{1})$
and $H_{0}^{1}(\omega _{2})$ respectively, in the other hand we have $%
f=\sum\limits_{i=0}^{m}f_{i}^{(1)}\times f_{i}^{(2)}$ for some $m\in 
\mathbb{N}
$ and $f_{i}^{(1)}\in H_{0}^{1}(\omega _{1}),$ $f_{i}^{(2)}\in
H_{0}^{1}(\omega _{2})$ for $i=0,...,m$, then we set, for every $n\in 
\mathbb{N}
:$ 

$$V_{n}^{(1)}:=span(e_{0}^{(1)},...,e_{n}^{(1)},f_{0}^{(1)},...,f_{m}^{(1)}),$$ 
$$V_{n}^{(2)}:=span(e_{0}^{(2)},...,e_{n}^{(2)},f_{0}^{(2)},...,f_{m}^{(2)}).$$
Now, since $f$ belongs to each $V_{n}^{(1)}\otimes V_{n}^{(2)}$ then
according to Lemma \ref{errorest}
one has, for every $\epsilon \in (0,1],$ $n\in 
\mathbb{N}
:$ 
\begin{multline*}
\left\Vert \nabla _{X_{2}}u_{\epsilon ,V_{n},f}-\nabla
_{X_{2}}u_{f}\right\Vert _{L^{2}(\Omega )^{N-q}}\leq \epsilon \left(
C_{1}C_{3}\left\Vert \nabla _{X_{1}}f\right\Vert _{L^{2}(\Omega
)^{q}}+C_{2}\left\Vert f\right\Vert _{L^{2}(\Omega )}\right) \\
+\frac{\left\Vert A_{22}\right\Vert_{L^{\infty}(\Omega )}}{\lambda}\inf_{v\in V_{n}}\left\Vert \nabla
_{X_{2}}v-\nabla _{X_{2}}u_{f}\right\Vert _{L^{2}(\Omega )^{N-q}},
\end{multline*}%
where $V_{n}:=V_{n}^{(1)}\otimes V_{n}^{(2)}$. According to Corollary \ref{cordensity} in Appendix \ref{appendixA}, $\cup
_{n\geq 0}(V_{n}^{(1)}\otimes V_{n}^{(2)})$ is dense in $H_{0}^{1}(\Omega )$%
. Using the fact that the sequence $(V_{n})_{n\geq 0}$ is nondecreasing,
then we obtain that 
\begin{equation*}
\forall \epsilon \in (0,1]:\lim_{n\rightarrow \infty }\inf_{v\in V_{n}}\Vert
\nabla v-\nabla u_{\epsilon ,f}\Vert _{L^{2}(\Omega )^{N}}=0,
\end{equation*}%
and therefore, by using (\ref{Cea2}) we get
\begin{equation*}
\forall \epsilon \in (0,1]:\lim_{n\rightarrow \infty }\Vert \nabla
u_{\epsilon ,V_{n},f}-\nabla u_{\epsilon ,f}\Vert _{L^{2}(\Omega )^{N}}=0,
\end{equation*}%
and thus 
\begin{equation*}
\forall \epsilon \in (0,1]:\lim_{n\rightarrow \infty }\Vert \nabla
_{X_{2}}u_{\epsilon ,V_{n},f}-\nabla _{X_{2}}u_{\epsilon ,f}\Vert
_{L^{2}(\Omega )^{N-q}}=0\text{, and }\lim_{n\rightarrow \infty }\Vert
\nabla _{X_{1}}u_{\epsilon ,V_{n},f}-\nabla _{X_{1}}u_{\epsilon ,f}\Vert
_{L^{2}(\Omega )^{q}}=0.
\end{equation*}
Using the fact that $H_{0}^{1}(\Omega )$ is dense in $H_{0}^{1}(\Omega
,\omega _{2})$ (Lemma \ref{lem:density}, Appendix \ref{appendixA}) and that the embedding $H_{0}^{1}(\Omega )\hookrightarrow H_{0}^{1}(\Omega ,\omega _{2})$
is continuous then $\cup _{n\geq 0}(V_{n}^{(1)}\otimes V_{n}^{(2)})$ is
dense in $H_{0}^{1}(\Omega ,\omega _{2})$. Using the fact that the sequence $%
(V_{n})_{n\geq 0}$ is nondecreasing, then we obtain that 
\begin{equation*}
\lim_{n\rightarrow \infty }\inf_{v\in V_{n}}\left\Vert \nabla _{X_{2}}v-\nabla _{X_{2}}u_{f}\right\Vert
_{L^{2}(\Omega )^{N-q}}=0.
\end{equation*}%
Now, passing to the limit, as $n\rightarrow \infty $, in the above
inequality we deduce
\begin{equation*}
\forall \epsilon \in (0,1]:\left\Vert \nabla _{X_{2}}u_{\epsilon ,f}-\nabla _{X_{2}}u_{f}\right\Vert
_{L^{2}(\Omega )^{N-q}}\leq \epsilon \left( C_{1}C_{3}\left\Vert \nabla
_{X_{1}}f\right\Vert _{L^{2}(\Omega )^{q}}+C_{2}\left\Vert f\right\Vert
_{L^{2}(\Omega )}\right) .
\end{equation*}
Finally, by using the second inequality of Lemma \ref{errorest} we get 
\begin{equation*}
\forall \epsilon \in (0,1]:\left\Vert \nabla _{X_{1}}u_{\epsilon ,V_{n},f}\right\Vert _{L^{2}(\Omega
)^{q}}\leq \frac{1}{\sqrt{2}}\left( C_{1}C_{3}\left\Vert \nabla
_{X_{1}}f\right\Vert _{L^{2}(\Omega )^{q}}+C_{2}\left\Vert f\right\Vert
_{L^{2}(\Omega )}\right) +C_{3}\left\Vert \nabla _{X_{1}}f\right\Vert
_{L^{2}(\Omega )^{q}},
\end{equation*}%
and the passage to limit as $n\rightarrow \infty $ shows the second
estimation of the lemma.
\end{proof}

Now, we are able to give the proof of Theorem \ref{mainthm}. Let us
introduce the space 
\begin{equation*}
H_{0}^{1}(\Omega ;\omega _{1})=\left\{ v\in L^{2}(\Omega )~\text{such that}%
~\nabla _{X_{1}}v\in L^{2}(\Omega )^{q}\text{ and for a.e. }X_{2}\in \omega
_{2}, v(\cdot ,X_{2})\in H_{0}^{1}(\omega _{1})\right\} ,
\end{equation*}%
normed by the Hilbertian norm $\left\Vert \nabla _{X_{1}}(\cdot )\right\Vert
_{L^{2}(\Omega )^{q}}.$ We have the Poincare's inequality 
\begin{equation}
\left\Vert v\right\Vert _{L^{2}(\Omega )}\leq C_{\omega _{1}}\left\Vert
\nabla _{X_{1}}v\right\Vert _{L^{2}(\Omega )^{q}}\text{ for any }v\in
H_{0}^{1}(\Omega ;\omega _{1})  \label{sob-bis}
\end{equation}

Let $f\in L^{2}(\Omega )$ such that $(\ref{hypFad1})$ and (\ref{hypFad2}),
thus $f\in H_{0}^{1}(\Omega ;\omega _{1})$. According to Lemma \ref%
{lemdensitysob} of Appendix \ref{appendixA} $H_{0}^{1}(\omega _{1})\otimes
H_{0}^{1}(\omega _{2})$ is dense in $H_{0}^{1}(\Omega ),$ and according to
Remark \ref{rem:density} of Appendix \ref{appendixA} $H_{0}^{1}(\Omega )$ is dense in 
$H_{0}^{1}(\Omega ;\omega _{1}),$ then it follows that $H_{0}^{1}(\omega
_{1})\otimes H_{0}^{1}(\omega _{2})$ is dense in $H_{0}^{1}(\Omega ;\omega
_{1}),$ thanks to the continuous embedding $H_{0}^{1}(\Omega
)\hookrightarrow H_{0}^{1}(\Omega ;\omega _{1})$. Therefore, for $\delta >0$
there exists $g_{\delta }\in H_{0}^{1}(\omega _{1})\otimes H_{0}^{1}(\omega
_{2})$ such that 
\begin{equation}
\left\Vert \nabla _{X_{1}}(f-g_{\delta })\right\Vert _{L^{2}(\Omega
)^{q}}\leq \delta.  \label{Gdelta}
\end{equation}
Let $u_{\epsilon ,g_{\delta }}$ be the unique solution of (\ref{probcontweak}) with $f$ replaced by $g_{\delta }$. Testing with $u_{\epsilon ,f}-u_{\epsilon ,g_{\delta }}$ in the difference of the weak formulations (recall that $\beta=0$)
\begin{equation*}
\int_{\Omega }A_{\epsilon }\nabla (u_{\epsilon ,f}-u_{\epsilon ,g_{\delta
}})\cdot \nabla \varphi dx=\int_{\Omega }(f-g_{\delta })\varphi dx,\text{ }%
\forall \varphi \in H_{0}^{1}(\Omega ),
\end{equation*}%
we obtain%
\begin{equation*}
\left\Vert \nabla _{X_{2}}u_{\epsilon ,f}-\nabla _{X_{2}}u_{\epsilon
,g_{\delta }}\right\Vert _{L^{2}(\Omega )^{N-q}}\leq \frac{C_{\omega
_{1}}C_{\omega _{2}}}{\lambda }\delta ,\text{ and }\left\Vert \nabla
_{X_{1}}u_{\epsilon ,f}-\nabla _{X_{1}}u_{\epsilon ,g_{\delta }}\right\Vert
_{L^{2}(\Omega )^{q}}\leq \frac{C_{\omega _{1}}C_{\omega _{2}}}{\lambda
\epsilon }\delta ,
\end{equation*}%
where we have used the ellipticity assumption, Poincar%
\'{e}'s inequalities ($\ref{sob}$), (\ref{sob-bis}), and (\ref{Gdelta}). By passing to the limit as $\epsilon \rightarrow 0$ in the
first inequality above, using Theorem \ref{thrmconvergence}, we get 
\begin{equation*}
\left\Vert \nabla _{X_{2}}u_{f}-\nabla _{X_{2}}u_{g_{\delta }}\right\Vert
_{L^{2}(\Omega )^{N-q}}\leq \frac{C_{\omega _{1}}C_{\omega _{2}}}{\lambda }%
\delta .
\end{equation*}%
Applying Lemma \ref{lem:error} on $u_{\epsilon ,g_{\delta }}$ and $%
u_{g_{\delta }}$ we obtain 
\begin{equation*}
\left\Vert \nabla _{X_{2}}u_{\epsilon ,g_{\delta }}-\nabla
_{X_{2}}u_{g_{\delta }}\right\Vert _{L^{2}(\Omega )^{N-q}}\leq \epsilon
\left( C_{1}C_{3}\left\Vert \nabla _{X_{1}}g_{\delta }\right\Vert
_{L^{2}(\Omega )^{q}}+C_{2}\left\Vert g_{\delta }\right\Vert _{L^{2}(\Omega
)}\right),
\end{equation*}%
and from (\ref{Gdelta}) we derive%
\begin{equation*}
\left\Vert \nabla _{X_{2}}u_{\epsilon ,g_{\delta }}-\nabla
_{X_{2}}u_{g_{\delta }}\right\Vert _{L^{2}(\Omega )^{N-q}}\leq \epsilon
\left( C_{1}C_{3}(\left\Vert \nabla _{X_{1}}f\right\Vert _{L^{2}(\Omega
)^{q}}+\delta )+C_{2}\left\Vert g_{\delta }\right\Vert _{L^{2}(\Omega
)}\right).
\end{equation*}%
Notice that $\left\Vert g_{\delta }\right\Vert _{L^{2}(\Omega )}\rightarrow
\left\Vert f\right\Vert _{L^{2}(\Omega )}$ as $\delta \rightarrow 0$, thanks
to (\ref{Gdelta}) and Poincar\'{e}'s in\'{e}quality (\ref{sob-bis}).
Finally, the triangle inequality gives
\begin{multline*}
\left\Vert \nabla _{X_{2}}u_{\epsilon ,f}-\nabla _{X_{2}}u_{f}\right\Vert
_{L^{2}(\Omega )^{N-q}}\leq \left\Vert \nabla _{X_{2}}u_{\epsilon ,f}-\nabla
_{X_{2}}u_{\epsilon ,g_{\delta }}\right\Vert _{L^{2}(\Omega )^{N-q}} \\
+\left\Vert \nabla _{X_{2}}u_{\epsilon ,g_{\delta }}-\nabla
_{X_{2}}u_{g_{\delta }}\right\Vert _{L^{2}(\Omega )^{N-q}}+\left\Vert \nabla
_{X_{2}}u_{g_{\delta }}-\nabla _{X_{2}}u_{f}\right\Vert _{L^{2}(\Omega
)^{N-q}} \\
\leq \epsilon \left( C_{1}C_{3}(\left\Vert \nabla _{X_{1}}f\right\Vert
_{L^{2}(\Omega )^{q}}+\delta )+C_{2}\left\Vert g_{\delta }\right\Vert
_{L^{2}(\Omega )}\right) +2\frac{C_{\omega _{1}}C_{\omega _{2}}}{\lambda }%
\delta .
\end{multline*}%
Passing to the limit as $\delta \rightarrow 0$ we obtain%
\begin{equation*}
\left\Vert \nabla _{X_{2}}u_{\epsilon ,f}-\nabla _{X_{2}}u_{f}\right\Vert
_{L^{2}(\Omega )^{N-q}}\leq \epsilon \left( C_{1}C_{3}\left\Vert \nabla
_{X_{1}}f\right\Vert _{L^{2}(\Omega )^{q}}+C_{2}\left\Vert f\right\Vert
_{L^{2}(\Omega )}\right) ,
\end{equation*}%
which is the estimation given in Theorem \ref{mainthm}.

For the estimation in the first direction, we have 
\begin{eqnarray*}
\forall\epsilon\in(0,1]:\left\Vert \nabla _{X_{1}}u_{\epsilon ,f}\right\Vert _{L^{2}(\Omega )^{q}}
&\leq &\left\Vert \nabla _{X_{1}}u_{\epsilon ,f}-\nabla _{X_{1}}u_{\epsilon
,g_{\delta }}\right\Vert _{L^{2}(\Omega )^{q}}+\left\Vert \nabla
_{X_{1}}u_{\epsilon ,g_{\delta }}\right\Vert _{L^{2}(\Omega )^{q}} \\
&\leq &\frac{C_{\omega _{1}}C_{\omega _{2}}}{\lambda \epsilon }\delta +\frac{%
1}{\sqrt{2}}\left( C_{1}C_{3}\left\Vert \nabla _{X_{1}}g_{\delta
}\right\Vert _{L^{2}(\Omega )^{q}}+C_{2}\left\Vert g_{\delta }\right\Vert
_{L^{2}(\Omega )}\right) +C_{3}\left\Vert \nabla _{X_{1}}g_{\delta
}\right\Vert _{L^{2}(\Omega )^{q}},
\end{eqnarray*}
where we have applied the triangle inequality and Lemma \ref{lem:error}.
Passing to the limit as $\delta \rightarrow 0$ by using $(\ref{Gdelta})$, we obtain 
\begin{equation*}
\forall\epsilon\in(0,1]:\left\Vert \nabla _{X_{1}}u_{\epsilon ,f}\right\Vert _{L^{2}(\Omega
)^{q}}\leq \frac{1}{\sqrt{2}}\left( C_{1}C_{3}\left\Vert \nabla
_{X_{1}}f\right\Vert _{L^{2}(\Omega )^{q}}+C_{2}\left\Vert f\right\Vert
_{L^{2}(\Omega )}\right) +C_{3}\left\Vert \nabla _{X_{1}}f\right\Vert
_{L^{2}(\Omega )^{q}}.
\end{equation*}
Hence, passing to the limit in $L^{2}(\Omega)-weak$ as $\epsilon \rightarrow 0$, up to a
subsequence, we show that $u_{f}$ belongs to $H_{0}^{1}(\Omega )$, and by a
contradiction argument, using the metrizability (for the weak topology) of
weakly compact subsets in separable Hilbert spaces, one can show
that the global sequence $(\nabla _{X_{1}}u_{\epsilon ,f})_{\epsilon }$
converges weakly to $\nabla _{X_{1}}u_{f}$ in $L^{2}(\Omega )^{q},$ and this completes the proof of Theorem \ref{mainthm}.
\begin{rmk}\label{rem-semigroup4}
In the case $\beta(s)=\mu s$ with $\mu>0$, we repeat the same arguments of this subsection by using Remark \ref{rem-poursemigroup3}, and then we obtain the estimation of Remark \ref{rem-semigroup-importante}.
\end{rmk}

\section{Anisotropic Perturbations of semigroups}

\subsection{Preliminaries}

For the standard basic theory of semigroups of bounded linear operators, we refer the reader to \cite{Pazy}. Let us begin by some reminders. Let $E$ be a real Banach space. An unbounded linear operator $\mathcal{A}:D(\mathcal{A)\subset 
}E\rightarrow E$ is said to be closed if for every sequence $(x_{n})$ of $D(%
\mathcal{A)}$ such that $(x_{n})$ and $(\mathcal{A}(x_{n}))$ converge in $E,$
we have $\lim x_{n}\in D(\mathcal{A)}$ and $\lim \mathcal{A}(x_{n})=\mathcal{%
A}(\lim x_{n}).$ An operator is said to be densely defined on $E$ if its
domain $D(\mathcal{A)}$ is dense in $E.$ Let $\mu \in 
\mathbb{R}
$, we said that $\mu $ belongs to the resolvent set of $\mathcal{A}$ if $(\mu
I-\mathcal{A)}:D(\mathcal{A)\rightarrow }E$ is one-to-one and onto and such
that $R_{\mu }=(\mu I-\mathcal{A)}^{-1}:E\rightarrow D(\mathcal{A)\subset }E$
is a bounded operator on $E.$ Notice that $R_{\mu }$ and $\mathcal{A}$
commute on $D(\mathcal{A)}$, that is $\forall x\in $ $D(\mathcal{A)}:$ $%
R_{\mu }\mathcal{A}x=\mathcal{A}R_{\mu }x.$
Let $\mathcal{A}$ be a densely defined closed operator. The bounded operator 
\begin{equation*}
\mathcal{A}_{\mu }=\mu \mathcal{A}(\mu I-\mathcal{A})^{-1}=\mu \mathcal{A}%
R_{\mu }=\mu ^{2}R_{\mu }-\mu I,
\end{equation*}%
is called the Yosida approximation of $\mathcal{A}$. We check immediately
that $\mathcal{A}_{\mu }$ and $\mathcal{A}$ commute on $D(\mathcal{A)}$ that
is for every $x\in D(\mathcal{A)}$ we have $\mathcal{A}_{\mu }x\in D(%
\mathcal{A)}$ and $\mathcal{AA}_{\mu }x=\mathcal{A}_{\mu }\mathcal{A}x$. Furthermore,
since $\mathcal{A}$ is closed then $e^{t\mathcal{A}_{\mu }}$ and $\mathcal{A}
$ commute on $D(\mathcal{A})$, that is
\begin{equation}
\forall t\in 
\mathbb{R}
,\forall x\in D(\mathcal{A}),e^{t\mathcal{A}_{\mu }}x\in D(\mathcal{A}),
\label{abstact-commute}
\end{equation}%
and 
\begin{equation*}
\mathcal{A}e^{t\mathcal{A}_{\mu }}x=e^{t\mathcal{A}_{\mu }}\mathcal{A}%
x=\sum_{k=0}^{\infty }\frac{t^{k}}{k!}(\mathcal{A}_{\mu })^{k}\mathcal{A}x,
\end{equation*}%
indeed, we can check by induction that if $x\in D(\mathcal{A})$ then $(%
\mathcal{A}_{\mu })^{k}x\in D(\mathcal{A}),$ and that $(\mathcal{A}_{\mu
})^{k}$ and $\mathcal{A}$ commute on $D(\mathcal{A})$.
Recall also that if $(\mu I-\mathcal{A)}^{-1}$ exists for $\mu >0$ and such
that $\left\Vert (\mu I-\mathcal{A)}^{-1}\right\Vert \leq \frac{1}{\mu }$
then 
\begin{equation*}
\forall t\geq 0:\left\Vert e^{t\mathcal{A}_{\mu }}\right\Vert =\left\Vert
e^{t\mu ^{2}R_{\mu }}\right\Vert \times \left\Vert e^{-\mu tI}\right\Vert
\leq e^{t\mu ^{2}\left\Vert R_{\mu }\right\Vert }\times e^{-\mu t}\leq 1,
\end{equation*}
where $\left\Vert \cdot \right\Vert $ is the operator norm of $\mathcal{L}%
(E)$.
A $C_{0}$ semigroup of bounded linear operators on $E$ is a family of
bounded operators $(S(t))_{t\geq 0}$ of $\mathcal{L}(E)$ such that: $S(0)=I$%
, for every $t,s\geq 0:$ $S(t+s)=S(t)S(s)$, and for every $x\in E:\left\Vert
S(t)x-x\right\Vert _{E}\rightarrow 0$ as $t\rightarrow 0.$ $(S(t))_{t\geq 0}$
is called a semigroup of contractions if for every $t\geq 0:\left\Vert
S(t)\right\Vert _{E}\leq 1$. Now, let us recall the well-known Hill-Yosida
theorem in its Hilbertian (real) version: An unbounded operator $\mathcal{A}$
is the infinitesimal generator of a $C_{0}$ semigroup of contractions $%
(S(t))_{t\geq 0}$ if and only if $\mathcal{A}$ is maximal dissipative, that
is when $\mu I-\mathcal{A}$ is surjective for every $\mu >0$ and for every $%
x\in D(\mathcal{A}$ $):$ $\left\langle \mathcal{A}x,x\right\rangle \leq 0$.
Recall that, in this case $D(\mathcal{A}$ $)$ is dense and $\mathcal{A}$ is
closed and its resolvent set contains $]0,+\infty[$. Furthermore, for every $t\geq 0,$ $e^{t\mathcal{A}_{\mu }}$
converges, in the strong operator topology, to $S(t)$, as $\mu \rightarrow
+\infty $ i.e. $\forall x\in E:e^{t\mathcal{A}_{\mu }}x\rightarrow S(t)x$
in $E$ as $\mu \rightarrow +\infty .$

Let $\Omega $ as in the introduction. The basic Hilbert space in the sequel
is $E=L^{2}(\Omega ).$ For $\epsilon \in (0,1],$ we introduce the
operator $\mathcal{A}_{\epsilon }$ acting on $L^{2}(\Omega )$ and given by
the formula 
\begin{equation*}
\mathcal{A}_{\epsilon }u=\text{div}(A_{\epsilon }\nabla u),
\end{equation*}%
where $A_{\epsilon }$ is given as in the introduction of this paper.
The domain of $\mathcal{A}_{\epsilon }$ is given by 
\begin{equation*}
D(\mathcal{A}_{\epsilon })=\{u\in H_{0}^{1}(\Omega )\mid \text{div}%
(A_{\epsilon }\nabla u)\in L^{2}(\Omega )\},
\end{equation*}%
where div$(A_{\epsilon }\nabla u)\in L^{2}(\Omega )$ is taken in the
distributional sense. Now, we introduce the operator $\mathcal{A}_{0}$
defined on 
\begin{equation*}
D(\mathcal{A}_{0})=\{u\in H_{0}^{1}(\Omega ;\omega _{2})\mid \text{div}%
_{X_{2}}(A_{22}\nabla _{X_{2}}u)\in L^{2}(\Omega )\},
\end{equation*}%
by the formula 
\begin{equation*}
\mathcal{A}_{0}u=\text{div}_{X_{2}}(A_{22}\nabla _{X_{2}}u).
\end{equation*}%
We check immediately, by using assumptions $(\ref{hypA1}-\ref{hypA2})$, that 
$\mathcal{A}_{\epsilon }$ and $\mathcal{A}_{0}$ are maximal dissipative and
therefore, they are the infinitesimal generators of $C_{0}$ semigroups of
contractions on $L^{2}(\Omega ),$ denoted $\left( S_{\epsilon }(t)\right)
_{t\geq 0}$ and $\left( S_{0}(t)\right) _{t\geq 0}$ respectively. For $\mu >0$ we denote by $R_{\epsilon ,\mu }$ the resolvent of $\mathcal{A}%
_{\epsilon }$. Similarly, we denote by $R_{0,\mu }$ the resolvent of $\mathcal{A}%
_{0}$. For $f\in L^{2}(\Omega ),$ we denote $u_{\epsilon ,\mu }$
the unique solution in $H_{0}^{1}(\Omega )$ to 
\begin{equation*}
\mu \int_{\Omega }u_{\epsilon ,\mu }\varphi dx+\int_{\Omega }A_{\epsilon
}\nabla u_{\epsilon ,\mu }\cdot \nabla \varphi dx=\int_{\Omega }f\text{ }%
\varphi dx\text{, }\forall \varphi \in H_{0}^{1}(\Omega ),\text{\ }
\end{equation*}%
we have $R_{\epsilon ,\mu }f=u_{\epsilon ,\mu }$ and $\left\Vert R_{\epsilon
,\mu }\right\Vert \leq \frac{1}{\mu }$, where $\left\Vert \cdot \right\Vert $
is the operator norm of $\mathcal{L}(L^{2}(\Omega )).$ Similarly, let $%
u_{0,\mu }$ be the unique solution in $H_{0}^{1}(\Omega ;\omega _{2})$ to 
\begin{equation}
\mu \int_{\Omega }u_{0,\mu }\varphi dx+\int_{\Omega }A_{22}\nabla
_{X_{2}}u_{0,\mu }\cdot \nabla _{X_{2}}\varphi dx=\int_{\Omega }f\text{ }%
\varphi dx\text{, }\forall \varphi \in H_{0}^{1}(\Omega ;\omega _{2}),
\label{resolvent0}
\end{equation}%
we \ have $R_{0,\mu }f=u_{0,\mu }$ and $\left\Vert R_{0,\mu }\right\Vert
\leq \frac{1}{\mu }.$
According to Remark \ref{rem-semigroup-importante}, we have the following 

\begin{lem}
Assume $(\ref{hypA1})$, $(\ref{hypA2})$, $(\ref{hypAd1})$, $(\ref{hypAd2})$
and $(\ref{hypA12-seconde}).$ Let $f\in H_{0}^{1}(\Omega ;\omega _{1})$, then
there exists $C_{A,\Omega }>0$ depending only on $A$ and $\Omega .$ such
that: 
\begin{equation}
\text{ }\forall \epsilon \in (0,1]\text{, }\forall \mu >0:\text{ }\left\Vert
R_{\epsilon ,\mu }f-R_{0,\mu }f\right\Vert _{L^{2}(\Omega )}\leq C_{A,\Omega
}\times \frac{\epsilon }{\mu }\times \left( \left\Vert \nabla
_{X_{1}}f\right\Vert _{L^{2}(\Omega )}+\left\Vert f\right\Vert
_{L^{2}(\Omega )}\right) .  \label{estm-resolvent}
\end{equation}
\end{lem}

\subsection{The asymptotic behavior of the perturbed semigroup}

In this subsection, we study the relationship between the semigroups $\left(
S_{\epsilon }(t)\right) _{t\geq 0}$ and $\left( S_{0}(t)\right) _{t\geq 0}$.
We will assume that 
\begin{equation}
A\in W^{1,\infty}(\Omega)^{N^{2}}.   \label{A-liptschitz}
\end{equation}%
Notice that (\ref{A-liptschitz}) shows that, for any $\epsilon \in
(0,1]$: 
\begin{equation*}
H_{0}^{1}(\Omega )\cap H^{2}(\Omega )\subset D(\mathcal{A}_{0})\cap D(%
\mathcal{A}_{\epsilon }).
\end{equation*}
Remark also that (\ref{A-liptschitz}) implies $(\ref{hypAd1})$. Now, we can give the main theorem of this section.
\begin{thm}
\label{thrm-semigroups}Let $\Omega =\omega _{1}\times \omega _{2}$
be a bounded domain of $%
\mathbb{R}
^{q}\times 
\mathbb{R}
^{N-q}.$ Assume $(\ref{hypA1})$, $(\ref{hypA2})$, $($\ref{hypAd2}$),$ $(\ref%
{hypA12-seconde})$ and $(\ref{A-liptschitz})$. Let $g\in
L^{2}(\Omega )$ and $T\geq 0$, we have: 
\begin{equation*}
\sup_{t\in \lbrack 0,T]}\left\Vert S_{\epsilon }(t)g-S_{0}(t)g\text{ }%
\right\Vert _{L^{2}(\Omega )}\rightarrow 0\text{ as }\epsilon \rightarrow 0.
\end{equation*}%
In particular, for $g\in \left( H_{0}^{1}\cap H^{2}(\omega _{1})\right)
\otimes (H_{0}^{1}\cap H^{2}(\omega _{2}))$ there exists $C_{g,A,\Omega }>0$ such that : 
\begin{equation*}
\forall\epsilon\in(0,1]:\sup_{t\in \lbrack 0,T]}\left\Vert S_{\epsilon }(t)g-S_{0}(t)g\right\Vert
_{L^{2}(\Omega )}\leq C_{g,A,\Omega }\times T\times \epsilon.
\end{equation*}
\end{thm}
Let us begin by this important lemma

\begin{lem}
\label{lem-yosida-approx} Suppose that assumptions of Theorem \ref%
{thrm-semigroups} hold. Let $f\in H_{0}^{1}(\Omega )\cap D(\mathcal{A}_{0})$
such that 

$$\text{div}_{X_{1}}(A_{11}\nabla _{X_{1}}f),\text{ div}_{X_{1}}(A_{12}\nabla
_{X_{2}}f),\text{ div}_{X_{2}}(A_{21}\nabla _{X_{1}}f)\in L^{2}(\Omega ) \text{, and } \mathcal{A}_{0}f\in H_{0}^{1}(\Omega ;\omega _{1}),$$ then there exists a constant $C_{f,A,\Omega }>0$ such that for every $\mu >0,$ $\epsilon \in
(0,1]$ we have: 
\begin{equation*}
\left\Vert \mathcal{A}_{\epsilon ,\mu }f-\mathcal{A}_{0,\mu }f\right\Vert
_{L^{2}(\Omega )}\leq C_{f,A,\Omega }\times \epsilon ,
\end{equation*}%
where $\mathcal{A}_{\epsilon ,\mu }$ and $\mathcal{A}_{0,\mu }$ are the
Yosida approximations of $\mathcal{A}_{\epsilon }$ and $\mathcal{A}_{0}$
respectively. The constant $C_{f,A,\Omega }$ is given by:
\begin{multline*}
C_{f,A,\Omega }=\left\Vert \text{div}_{X_{1}}(A_{11}\nabla
_{X_{1}}f)\right\Vert _{L^{2}(\Omega )}+\left\Vert \text{div}%
_{X_{1}}(A_{12}\nabla _{X_{2}}f)\right\Vert _{L^{2}(\Omega )} \\
+\left\Vert \text{div}_{X_{2}}(A_{21}\nabla _{X_{1}}f)\right\Vert
_{L^{2}(\Omega )}+C_{A,\Omega }\left( \left\Vert \nabla _{X_{1}}\mathcal{A}%
_{0}f\right\Vert _{L^{2}(\Omega )}+\left\Vert \mathcal{A}_{0}f\right\Vert
_{L^{2}(\Omega )}\right). 
\end{multline*}
\end{lem}

\begin{proof}
Let $\epsilon \in (0,1]$ and $\mu >0.$ The bounded operators $\mathcal{A}%
_{\epsilon ,\mu },$ $\mathcal{A}_{0,\mu }$ of $\mathcal{L(}L^{2}(\Omega ))$
are given by:
$$\mathcal{A}_{\epsilon ,\mu }=\mu \mathcal{A}_{\epsilon }R_{\epsilon ,\mu }\ \text{and}\  \mathcal{A}_{0,\mu }=\mu \mathcal{A}_{0}R_{0,\mu}.$$
Now, under the above hypothesis we obtain that $f\in D(\mathcal{A}_{\epsilon
})\cap D(\mathcal{A}_{0})$. We have:
\begin{eqnarray*}
\left\Vert \mathcal{A}_{\epsilon ,\mu }f-\mathcal{A}_{0,\mu }f\right\Vert
_{L^{2}(\Omega )} &=&\mu \left\Vert \mathcal{A}_{\epsilon }R_{\epsilon ,\mu
}f-\mathcal{A}_{0}R_{0,\mu }f\right\Vert _{L^{2}(\Omega )}=\mu \left\Vert
R_{\epsilon ,\mu }\mathcal{A}_{\epsilon }f-R_{0,\mu }\mathcal{A}%
_{0}f\right\Vert _{L^{2}(\Omega )} \\
&\leq &\mu \left\Vert R_{\epsilon ,\mu }\mathcal{A}_{\epsilon }f-R_{\epsilon
,\mu }\mathcal{A}_{0}f\right\Vert _{L^{2}(\Omega )}+\mu \left\Vert
R_{\epsilon ,\mu }\mathcal{A}_{0}f-R_{0,\mu }\mathcal{A}_{0}f\right\Vert
_{L^{2}(\Omega )} \\
&\leq &\mu \left\Vert R_{\epsilon ,\mu }\right\Vert \times \left\Vert 
\mathcal{A}_{\epsilon }f-\mathcal{A}_{0}f\right\Vert _{L^{2}(\Omega )}+\mu
\left\Vert R_{\epsilon ,\mu }\mathcal{A}_{0}f-R_{0,\mu }\mathcal{A}%
_{0}f\right\Vert _{L^{2}(\Omega )}.
\end{eqnarray*}
Since $\mathcal{A}_{0}f\in H_{0}^{1}(\Omega ;\omega _{1})$ by hypothesis,
then by using (\ref{estm-resolvent}) (where we replace $f$ by $\mathcal{A}%
_{0}f)$ and the fact that $\left\Vert R_{\epsilon ,\mu }\right\Vert \leq 
\frac{1}{\mu }$, we obtain 
\begin{eqnarray*}
\left\Vert \mathcal{A}_{\epsilon ,\mu }f-\mathcal{A}_{0,\mu }f\right\Vert
_{L^{2}(\Omega )} &\leq &\left\Vert \mathcal{A}_{\epsilon }f-\mathcal{A}%
_{0}f\right\Vert _{L^{2}(\Omega )}+\epsilon C_{A,\Omega }\left( \left\Vert
\nabla _{X_{1}}\mathcal{A}_{0}f\right\Vert _{L^{2}(\Omega )}+\left\Vert 
\mathcal{A}_{0}f\right\Vert _{L^{2}(\Omega )}\right)  \\
&=&\epsilon \left( 
\begin{array}{c}
\epsilon \left\Vert \text{div}_{X_{1}}(A_{11}\nabla _{X_{1}}f)\right\Vert
_{L^{2}(\Omega )}+\left\Vert \text{div}_{X_{1}}(A_{12}\nabla
_{X_{2}}f)\right\Vert _{L^{2}(\Omega )} \\ 
+\left\Vert \text{div}_{X_{2}}(A_{21}\nabla _{X_{1}}f)\right\Vert
_{L^{2}(\Omega )}+C_{A,\Omega }\left( \left\Vert \nabla _{X_{1}}\mathcal{A}%
_{0}f\right\Vert _{L^{2}(\Omega )}+\left\Vert \mathcal{A}_{0}f\right\Vert
_{L^{2}(\Omega )}\right) 
\end{array}%
\right)  \\
&\leq &C_{f,A,\Omega }\times \epsilon,
\end{eqnarray*}
where we have used the identity:%
\begin{equation*}
\mathcal{A}_{\epsilon }f-\mathcal{A}_{0}f=\epsilon ^{2}\text{div}%
_{X_{1}}(A_{11}\nabla _{X_{1}}f)+\epsilon \text{div}_{X_{1}}(A_{12}\nabla
_{X_{2}}f)+\epsilon \text{div}_{X_{2}}(A_{21}\nabla _{X_{1}}f),
\end{equation*}%
and the proof of the lemma is finished.
\end{proof}

\begin{lem}
\label{lem-yosida-error}Under assumptions of Theorem \ref{thrm-semigroups},
we have for any $g\in \left( H_{0}^{1}\cap H^{2}(\omega _{1})\right) \otimes
(H_{0}^{1}\cap H^{2}(\omega _{2})):$ 
\begin{equation*}
\forall \mu >0,\forall t\geq 0,\forall \epsilon \in (0,1]:\left\Vert e^{t%
\mathcal{A}_{\epsilon ,\mu }}g-e^{t\mathcal{A}_{0,\mu }}g\right\Vert
_{L^{2}(\Omega )}\leq C_{g,A,\Omega }\times t\times \epsilon ,
\end{equation*}%
where $C_{g,A,\Omega }$ is independent of $\mu $ and $\epsilon .$
\end{lem}

\begin{proof}
Let $\mu >0$ and $t\geq 0$ and $\epsilon \in (0,1]$ , we have 
\begin{eqnarray*}
e^{t\mathcal{A}_{0,\mu }}-e^{t\mathcal{A}_{\epsilon ,\mu }} &=&\int_{0}^{t}%
\frac{d}{ds}\left( e^{(t-s)\mathcal{A}_{\epsilon ,\mu }}e^{s\mathcal{A}%
_{0,\mu }}\right) ds \\
&=&\int_{0}^{t}e^{(t-s)\mathcal{A}_{\epsilon ,\mu }}(\mathcal{A}_{0,\mu }-%
\mathcal{A}_{\epsilon ,\mu })e^{s\mathcal{A}_{0,\mu }}ds.
\end{eqnarray*}
Hence, for $g\in L^{2}(\Omega )$ we have 
\begin{equation}
\left\Vert e^{t\mathcal{A}_{\epsilon ,\mu }}g-e^{t\mathcal{A}_{0,\mu
}}g\right\Vert _{L^{2}(\Omega )}\leq \int_{0}^{t}\left\Vert \mathcal{A}%
_{0,\mu }e^{s\mathcal{A}_{0,\mu }}g-\mathcal{A}_{\epsilon ,\mu }e^{s\mathcal{%
A}_{0,\mu }}g\right\Vert _{L^{2}(\Omega )}ds,  \label{est-yosida}
\end{equation}%
where have used $\left\Vert e^{(t-s)\mathcal{A}_{\epsilon ,\mu
}}\right\Vert \leq 1,$ since $t-s\geq 0.$

Now, we suppose that $g\in \left( H_{0}^{1}\cap H^{2}(\omega _{1})\right)
\otimes (H_{0}^{1}\cap H^{2}(\omega _{2}))$ (remark that $g\in D(\mathcal{A}%
_{0})$). For $s\geq0$ and $\mu>0$ we set: 
\begin{equation*}
f_{g,s,\mu}:=e^{s\mathcal{A}_{0,\mu }}g
\end{equation*}%
We can prove that $f_{g,s,\mu}$ fulfills the same hypothesis satisfied by the
function $f$ of Lemma \ref{lem-yosida-approx}. Moreover, for every $%
i,j=1,...,q$ we have $D_{x_{i}x_{j}}^{2}f_{g,s,\mu}\in L^{2}(\Omega)$ with:    
\begin{equation}
\left\Vert D_{x_{i}x_{j}}^{2}f_{g,s,\mu}\right\Vert _{L^{2}(\Omega )}\leq
\left\Vert D_{x_{i}x_{j}}^{2}g\right\Vert _{L^{2}(\Omega )}\text{ , }%
\left\Vert D_{x_{i}}f_{g,s,\mu}\right\Vert _{L^{2}(\Omega )}\leq \left\Vert
D_{x_{i}}g\right\Vert _{L^{2}(\Omega )}, \label{control1}
\end{equation}%
and
\begin{equation}
\left\Vert\left( \mathcal{A}_{0}f_{g,s,\mu}\right) \right\Vert
_{L^{2}(\Omega )}\leq \left\Vert \mathcal{A}_{0}g\right\Vert
_{L^{2}(\Omega )} \text{, }\left\Vert D_{x_{i}}(\mathcal{A}_{0}f_{g,s,\mu})\right\Vert _{L^{2}(\Omega )}\leq
\left\Vert D_{x_{i}}(\mathcal{A}_{0}g)\right\Vert _{L^{2}(\Omega )},
\label{control2}
\end{equation}%
also for every $i=1,...,q$, $j=q+1,...,N$ we have $D_{x_{i}x_{j}}^{2}f_{g,s,\mu}\in L^{2}(\Omega ) $ with :
\begin{equation}
\left\Vert D_{x_{j}}f_{g,s,\mu}\right\Vert _{L^{2}(\Omega )}^{2}\leq \frac{1}{%
\lambda}\left\Vert \mathcal{A}_{0}g\right\Vert _{L^{2}(\Omega
)}\left\Vert g\right\Vert _{L^{2}(\Omega )}\ \text{and} \ \left\Vert
D_{x_{i}x_{j}}^{2}f_{g,s,\mu}\right\Vert^{2} _{L^{2}(\Omega )}\leq \frac{1}{\lambda}%
\left\Vert D_{x_{i}}\mathcal{A}_{0}g\right\Vert _{L^{2}(\Omega )}\left\Vert
D_{x_{i}}g\right\Vert _{L^{2}(\Omega )}. \label{control3}
\end{equation}
The proof of these assertions follows from the identity $e^{s\mathcal{A}%
_{0,\mu }}(g_{1}\otimes g_{2})=g_{1}\otimes e^{s\mathcal{A}_{0,\mu }}g_{2}$
(see Appendix \ref{AppendixB}). Notice that the above bounds are independent of $s$, $\epsilon $, and $\mu$. \\
Now, apply Lemma \ref{lem-yosida-approx}, we get 
\begin{eqnarray*}
\left\Vert \mathcal{A}_{0,\mu }e^{s\mathcal{A}_{0,\mu }}g-\mathcal{A}%
_{\epsilon ,\mu }e^{s\mathcal{A}_{0,\mu }}g\right\Vert _{L^{2}(\Omega )}
&\leq &\epsilon \left( 
\begin{array}{c}
\left\Vert \text{div}_{X_{1}}(A_{11}\nabla
_{X_{1}}f_{g,s,\mu})\right\Vert _{L^{2}(\Omega )}+\left\Vert \text{div}%
_{X_{1}}(A_{12}\nabla _{X_{2}}f_{g,s,\mu})\right\Vert _{L^{2}(\Omega )} \\ 
+\left\Vert \text{div}_{X_{2}}(A_{21}\nabla _{X_{1}}f_{g,s,\mu})\right\Vert
_{L^{2}(\Omega )}+ \\ C_{A,\Omega }\left( \left\Vert \nabla _{X_{1}}\mathcal{A}%
_{0}f_{g,s,\mu}\right\Vert _{L^{2}(\Omega )}+\left\Vert \mathcal{A}%
_{0}f_{g,s,\mu}\right\Vert _{L^{2}(\Omega )}\right)
\end{array}%
\right).
\end{eqnarray*}
By using $(\ref{control1}-\ref{control3})$ with (\ref{A-liptschitz}), one can show that the quantity in parentheses in the above inequality is bounded by some $C_{g,A,\Omega}>0$ independent of $s$, $\epsilon $, and $\mu$, thus 
$$ \left\Vert \mathcal{A}_{0,\mu }e^{s\mathcal{A}_{0,\mu }}g-\mathcal{A}%
_{\epsilon ,\mu }e^{s\mathcal{A}_{0,\mu }}g\right\Vert _{L^{2}(\Omega )}\leq C_{g,A,\Omega} \times \epsilon.$$
Finally, integrate the above inequality in $s$ over $[0,t]$, and use (%
\ref{est-yosida}), we get the desired result.
\end{proof}

\bigskip Now, we are able to prove Theorem \ref{thrm-semigroups}. First we
prove the case when $g\in \left( H_{0}^{1}\cap H^{2}(\omega _{1})\right)
\otimes (H_{0}^{1}\cap H^{2}(\omega _{2}))$ and we conclude by a density
argument. So, let $g$ as mentioned above, by Lemma \ref{lem-yosida-error} we
have 
\begin{equation}
\forall \mu >0,\forall t\geq 0,\forall \epsilon \in (0,1]:\left\Vert e^{t%
\mathcal{A}_{\epsilon ,\mu }}g-e^{t\mathcal{A}_{0,\mu }}g\right\Vert
_{L^{2}(\Omega )}\leq C_{g,A,\Omega}\times t\times \epsilon.   \label{yosida-error}
\end{equation}%
Passing to the limit in (\ref{yosida-error}) as $\mu
\rightarrow +\infty $ we get (see the preliminaries, the abstract part) 
\begin{equation*}
\forall t\geq 0,\forall \epsilon \in (0,1]:\left\Vert S_{\epsilon
}(t)g-S_{0}(t)g\right\Vert _{L^{2}(\Omega )}\leq C_{g,A,\Omega}\times t\times
\epsilon ,
\end{equation*}%
whence for $T\geq 0$ fixed we obtain 
\begin{equation}
\forall \epsilon \in (0,1]:\sup_{t\in \lbrack 0,T]}\left\Vert S_{\epsilon
}(t)g-S_{0}(t)g\right\Vert _{L^{2}(\Omega )}\leq C_{g,A,\Omega}\times T\times
\epsilon .  \label{taux-convergence}
\end{equation}%
Whence 
\begin{equation}
\sup_{t\in \lbrack 0,T]}\left\Vert S_{\epsilon }(t)g-S_{0}(t)g\right\Vert
_{L^{2}(\Omega )}\rightarrow 0\text{ as }\epsilon \rightarrow 0.
\label{yosida limit}
\end{equation}

Now, let $g\in L^{2}(\Omega )$ and let $\delta >0$, by density there exists $%
g_{\delta }\in \left( H_{0}^{1}\cap H^{2}(\omega _{1})\right) \otimes
(H_{0}^{1}\cap H^{2}(\omega _{2})$ such that 
\begin{equation*}
\left\Vert g-g_{\delta }\right\Vert _{L^{2}(\Omega )}\leq \frac{\delta }{4}.
\end{equation*}%
According to (\ref{yosida limit}) there exists $\epsilon _{\delta }>0$ such
that%
\begin{equation*}
\forall \epsilon \in (0,\epsilon _{\delta }]:\sup_{t\in \lbrack
0,T]}\left\Vert S_{\epsilon }(t)g_{\delta }-S_{0}(t)g_{\delta }\right\Vert
_{L^{2}(\Omega )}\leq \frac{\delta }{2}.
\end{equation*}%
Whence, by the triangle inequality we get 
\begin{equation*}
\forall \epsilon \in (0,\epsilon _{\delta }]:\sup_{t\in \lbrack
0,T]}\left\Vert S_{\epsilon }(t)g-S_{0}(t)g\right\Vert _{L^{2}(\Omega )}\leq 
\frac{\delta }{2}+\sup_{t\in \lbrack 0,T]}\left( \left\Vert S_{\epsilon
}(t)\right\Vert +\left\Vert S_{0}(t)\right\Vert \right) \times \left\Vert
g_{\delta }-g\right\Vert _{L^{2}(\Omega )}.
\end{equation*}%
Using the fact that the semigroups $\left( S_{\epsilon }(t)\right) _{t\geq 0}
$ and $\left( S_{0}(t)\right) _{t\geq 0}$ are of contractions, we get 
\begin{equation*}
\forall \epsilon \in (0,\epsilon _{\delta }]:\sup_{t\in \lbrack
0,T]}\left\Vert S_{\epsilon }(t)g-S_{0}(t)g\right\Vert _{L^{2}(\Omega )}\leq
\delta.
\end{equation*}%
So, $\sup_{t\in \lbrack 0,T]}\left\Vert S_{\epsilon
}(t)g-S_{0}(t)g\right\Vert _{L^{2}(\Omega )}\rightarrow 0$ as $\epsilon
\rightarrow 0.$ The second assertion of the theorem is given by (\ref%
{taux-convergence}) and the proof of Theorem \ref{thrm-semigroups} is completed.
\subsection{An application to linear parabolic equations}
Theorem \ref{thrm-semigroups} gives an opening for the study of anisotropic singular perturbations of evolution partial differential equations from the semigroup point of view. 
In this subsection, we give a simple and short application to the linear parabolic equation%
\begin{equation}
\frac{\partial u_{\epsilon }}{\partial t}-\text{div}(A_{\epsilon }\nabla
u_{\epsilon })=0,  \label{parabolic}
\end{equation}
supplemented with the boundary and the initial conditions  
\begin{eqnarray}
u_{\epsilon }(t,\cdot ) &=&0~\text{in}~\partial \Omega \text{  for }t\in
(0,+\infty )  \label{boundary-prabolic} \\
u_{\epsilon }(0,\cdot ) &=&u_{\epsilon ,0}.  \label{condition-initial}
\end{eqnarray}
The limit problem is 
\begin{equation}
\frac{\partial u}{\partial t}-\text{div}_{X_{2}}(A_{22}\nabla_{X_{2}} u)=0,
\label{parabolic-limit}
\end{equation}
supplemented with the boundary and the initial conditions 
\begin{eqnarray}
u(t,\cdot ) &=&0~\text{in}~\omega _{1}\times \partial \omega _{2}\text{ for }%
t\in (0,+\infty )  \label{bouidary-parabolic-limit} \\
u(0,\cdot ) &=&u_{0}.  \label{condition-initial-limit}
\end{eqnarray}
The operator forms of $(\ref{parabolic}-\ref{condition-initial})$ and $(\ref%
{parabolic-limit}-\ref{condition-initial-limit})$ read%
\begin{equation}
\frac{du_{\epsilon }}{dt}-\mathcal{A}_{\epsilon }u_{\epsilon }=0\text{, with 
}u_{\epsilon }(0)=u_{\epsilon ,0},  \label{parabolic-abstract}
\end{equation}%
and 
\begin{equation}
\frac{du}{dt}-\mathcal{A}_{0}u=0,\text{ with }u(0)=u_{0}.
\label{parabolic-abstract-limit}
\end{equation}
Suppose that $u_{0}\in D(\mathcal{A}_{0})$ and $u_{\epsilon ,0}\in D(%
\mathcal{A}_{\epsilon })$. Assume that ($\ref{hypA1}$), (\ref{hypA2}) hold, 
then it follows that (\ref{parabolic-abstract}), (\ref%
{parabolic-abstract-limit}) have unique classical solutions 

$$u_{\epsilon }\in C^{1}([0,+\infty );L^{2}(\Omega ))\cap C([0,+\infty );D(%
\mathcal{A}_{\epsilon })) \text{, and } u\in C^{1}([0,+\infty );L^{2}(\Omega ))\cap C([0,+\infty );D(\mathcal{A}%
_{0}))$$.

We have the following convergence result.

\begin{prop}\label{prop-parabolic}
Suppose that $u_{0}\in D(\mathcal{A}_{0})$ and $u_{\epsilon ,0}\in D(%
\mathcal{A}_{\epsilon })$ such that $u_{\epsilon ,0}\rightarrow u_{0}$ in $%
L^{2}(\Omega ),$ then under assumptions of Theorem \ref{thrm-semigroups}, we
have for any $T\geq 0$: 
\begin{equation}\label{parabolic-convergence}
\sup_{t\in \lbrack 0,T]}\left\Vert u_{\epsilon }(t)-u(t)\right\Vert
_{L^{2}(\Omega )}\rightarrow 0\text{ as }\epsilon \rightarrow 0.
\end{equation}
Moreover, if $u_{\epsilon ,0}$ and $u_{0}$ are in $H^{2}(\Omega )$ such
that $(u_{\epsilon ,0})$ is bounded in $H^{2}(\Omega )$ and $\Vert\nabla_{X_{2}}(u_{\epsilon ,0}-u_{0})\Vert_{L^2(\Omega)}\rightarrow 0 $,  $\Vert\nabla^{2}_{X_{2}}(u_{\epsilon ,0}-u_{0})\Vert_{L^2(\Omega)}\rightarrow 0 $ as $\epsilon\rightarrow 0$, then: 
\begin{equation*}
\sup_{t\in \lbrack 0,T]}\left\Vert \frac{d}{dt}(u_{\epsilon
}(t)-u(t))\right\Vert _{L^{2}(\Omega )}\rightarrow 0.
\end{equation*}
\end{prop}

\begin{proof}
It is well known that the solutions $u_{\epsilon }$, $u$ are given by 
\begin{equation*}
u_{\epsilon }(t)=S_{\epsilon }(t)u_{\epsilon ,0}\text{ and }u_{0
}(t)=S_{0}(t)u_{0}\text{, for every }t\geq 0.\text{ }
\end{equation*}%
Let $T\geq 0$, we have 
\begin{eqnarray*}
\sup_{t\in \lbrack 0,T]}\left\Vert u_{\epsilon }(t)-u(t)\right\Vert
_{L^{2}(\Omega )} &\leq &\sup_{t\in \lbrack 0,T]}\left\Vert S_{\epsilon
}(t)u_{\epsilon ,0}-S_{\epsilon }(t)u_{0}\right\Vert _{L^{2}(\Omega
)}+\sup_{t\in \lbrack 0,T]}\left\Vert S_{\epsilon
}(t)u_{0}-S_{0}(t)u_{0}\right\Vert _{L^{2}(\Omega )} \\
&\leq &\left\Vert u_{\epsilon ,0}-u_{0}\right\Vert _{L^{2}(\Omega
)}+\sup_{t\in \lbrack 0,T]}\left\Vert S_{\epsilon
}(t)u_{0}-S_{0}(t)u_{0}\right\Vert _{L^{2}(\Omega )}.
\end{eqnarray*}%
Passing to the limit as $\epsilon \rightarrow 0$ by using Theorem $\ref%
{thrm-semigroups}$, we get $\sup_{t\in \lbrack 0,T]}\left\Vert u_{\epsilon
}(t)-u(t)\right\Vert _{L^{2}(\Omega )}\rightarrow 0.$

For the second affirmation, we have: 
\begin{eqnarray*}
\left\Vert \frac{d}{dt}(u_{\epsilon }(t)-u(t))\right\Vert _{L^{2}(\Omega )}
&=&\left\Vert S_{\epsilon }(t)\mathcal{A}_{\epsilon }u_{\epsilon ,0}-S_{0}(t)%
\mathcal{A}_{0}u_{0}\right\Vert _{L^{2}(\Omega )} \\
&\leq &\left\Vert \mathcal{A}_{\epsilon }u_{\epsilon ,0}-\mathcal{A}%
_{0}u_{0}\right\Vert _{L^{2}(\Omega )}+\sup_{t\in \lbrack 0,T]}\left\Vert
S_{\epsilon }(t)\mathcal{A}_{0}u_{0}-S_{0}(t)\mathcal{A}_{0}u_{0}\right\Vert
_{L^{2}(\Omega )}.
\end{eqnarray*}

As $(u_{\epsilon ,0})$ is bounded in $H^{2}(\Omega )$, $u_{0}\in
H^{2}(\Omega )$ and $\Vert\nabla_{X_{2}}(u_{\epsilon ,0}-u_{0})\Vert_{L^2(\Omega)}\rightarrow 0 $,  $\Vert\nabla^{2}_{X_{2}}(u_{\epsilon ,0}-u_{0})\Vert_{L^2(\Omega)}\rightarrow 0 $ as $\epsilon\rightarrow 0$, then by using (\ref{A-liptschitz}) we get immediately  $\left\Vert \mathcal{A}%
_{\epsilon }u_{\epsilon ,0}-\mathcal{A}_{0}u_{0}\right\Vert _{L^{2}(\Omega
)}\rightarrow 0$ as $\epsilon \rightarrow 0$, and we conclude by applying
Theorem $\ref{thrm-semigroups}$. 
\end{proof}
\begin{rmk}
Consider the nonhomogeneous parabolic equations associated to (\ref{parabolic}) and (\ref{parabolic-limit}) with second member $f(t,x)$. Suppose that $f$ is regular enough, for example $f\in Lip([0,T];L^{2}(\Omega ))$, then the associated classical solutions $u_{\epsilon }$ and $u$ exist and they are unique. In this case, we have the same convergence result (\ref{parabolic-convergence}). The proof follows immediately from the use of the following integral representation formulas
$$u_{\epsilon }(t)=S_{\epsilon }(t)u_{\epsilon,0}+\int_{0}^{t}S_{\epsilon
}(t-r)f(r)dr \text{,    } u(t)=S_{0}(t)u_{0}+\int_{0}^{t}S_{0}(t-r)f(r)dr, \text{  } t\in[0,T],
$$%
Theorem \ref{thrm-semigroups}, and Lebesgue's theorem.
\end{rmk}
\section*{Acknowledgment}
The authors would like to thank Professor Robert Eymard for some useful discussions.

\bigskip \appendix

\section{Density lemmas}

\label{appendixA} Let $\omega _{1}$ and $\omega _{2}$ be two open bounded
subsets of $%
\mathbb{R}
^{q}$ and $%
\mathbb{R}
^{N-q}$ respectively. Recall that

\begin{equation*}
H_{0}^{1}(\Omega ;\omega _{2})=\left\{ u\in L^{2}(\Omega )\mid \nabla
_{X_{2}}u\in L^{2}(\Omega )^{N-q},\text{ and for a.e.}X_{1}\in \omega _{1},u(X_{1},\cdot
)\in H_{0}^{1}(\omega _{2})\right\} ,
\end{equation*}%
normed by $\left\Vert \nabla _{X_{2}}(\cdot )\right\Vert _{L^{2}(\Omega )}.$
We have the following

\begin{lem}
\label{lem:density}The space $H_{0}^{1}(\Omega )$ is dense in $%
H_{0}^{1}(\Omega ;\omega _{2}).$
\end{lem}

\begin{proof}
Let $u\in H_{0}^{1}(\Omega ;\omega _{2})$ fixed. Let $l$ be the linear form
defined on $H_{0}^{1}(\Omega )$ by%
\begin{equation*}
\forall \varphi \in H_{0}^{1}(\Omega ):l(\varphi )=\int_{\Omega }\nabla
_{X_{2}}u\cdot \nabla _{X_{2}}\varphi dx.
\end{equation*}%
$l$ is continuous on $H_{0}^{1}(\Omega )$, indeed we have 
\begin{equation*}
\forall \varphi \in H_{0}^{1}(\Omega ):\left\vert l(\varphi )\right\vert
\leq \left\Vert \nabla _{X_{2}}u\right\Vert _{L^{2}(\Omega )}\left\Vert
\nabla _{X_{2}}\varphi \right\Vert _{L^{2}(\Omega )},
\end{equation*}%
and then, 
\begin{equation*}
\forall \varphi \in H_{0}^{1}(\Omega ):\left\vert l(\varphi )\right\vert
\leq \left\Vert \nabla _{X_{2}}u\right\Vert _{L^{2}(\Omega )}\left\Vert
\nabla \varphi \right\Vert _{L^{2}(\Omega )}.
\end{equation*}%
For every $n\in 
\mathbb{N}
^{\ast },$ we denote $u_{n}$ the unique solution to 
\begin{equation}
\left\{ 
\begin{array}{c}
\frac{1}{n^{2}}\int_{\Omega }\nabla _{X_{1}}u_{n}\cdot \nabla
_{X_{1}}\varphi dx+\int_{\Omega }\nabla _{X_{2}}u_{n}\cdot \nabla
_{X_{2}}\varphi dx=l(\varphi )\text{, }\forall \varphi \in H_{0}^{1}(\Omega )
\\ 
u_{n}\in H_{0}^{1}(\Omega )\text{, \ \ \ \ \ \ \ \ \ \ \ \ \ \ \ \ \ \ \ \ \
\ \ \ \ \ \ \ \ \ \ \ \ \ \ \ \ \  \ \ \ \ \ \ \ \ \ \ 
\ \ \ \ \ \ \ \ \ \ \ \ \ \ \ \ \ \ \ \ \ \ }%
\end{array}%
\right.  \label{suite-Un}
\end{equation}%
where the existence and the uniqueness follow from the Lax-Milgram theorem. Testing with $u_{n}$ in (\ref{suite-Un}) we get, for every $n\in 
\mathbb{N}
^{\ast }$ 
\begin{equation*}
\frac{1}{n^{2}}\int_{\Omega }\left\vert \nabla _{X_{1}}u_{n}\right\vert
^{2}dx+\int_{\Omega }\left\vert \nabla _{X_{2}}u_{n}\right\vert ^{2}dx\leq
\left\Vert \nabla _{X_{2}}u\right\Vert _{L^{2}(\Omega )}\left\Vert \nabla
_{X_{2}}u_{n}\right\Vert _{L^{2}(\Omega )},
\end{equation*}%
then, we deduce that 
\begin{equation}
\forall n\in 
\mathbb{N}
^{\ast }:\left\Vert \nabla _{X_{2}}u_{n}\right\Vert _{L^{2}(\Omega )}\leq
\left\Vert \nabla _{X_{2}}u\right\Vert _{L^{2}(\Omega )},  \label{borne1}
\end{equation}%
and 
\begin{equation}
\forall n\in 
\mathbb{N}
^{\ast }:\frac{1}{n}\left\Vert \nabla _{X_{1}}u_{n}\right\Vert
_{L^{2}(\Omega )}\leq \left\Vert \nabla _{X_{2}}u\right\Vert _{L^{2}(\Omega
)}.  \label{borne2}
\end{equation}%
Using (\ref{borne1}) and Poincar\'{e}'s inequality we obtain:%
\begin{equation}
\forall n\in 
\mathbb{N}
^{\ast }:\left\Vert u_{n}\right\Vert _{L^{2}(\Omega )}\leq C_{\omega
_{2}}\left\Vert \nabla _{X_{2}}u\right\Vert _{L^{2}(\Omega )}.  \label{borne3}
\end{equation}%
Reflexivity of $L^{2}(\Omega )$ shows that there exists, $u_{\infty
},u_{\infty }^{\prime },u_{\infty }^{\prime \prime }\in L^{2}(\Omega )$ and
a subsequence still labeled $(u_{n})$ such that 
\begin{equation*}
u_{n}\rightharpoonup u_{\infty }\text{, }\nabla _{X_{2}}u_{n}\rightharpoonup
u_{\infty }^{\prime }\text{ and }\frac{1}{n}\nabla
_{X_{1}}u_{n}\rightharpoonup u_{\infty }^{\prime \prime }\text{ in }%
L^{2}(\Omega )\text{, weakly. }
\end{equation*}%
Using the continuity of derivation on $\mathcal{D}^{\prime }(\Omega )$  we get%
\begin{equation}
u_{n}\rightharpoonup u_{\infty }\text{, }\nabla _{X_{2}}u_{n}\rightharpoonup
\nabla _{X_{2}}u_{\infty }\text{ and }\frac{1}{n}\nabla
_{X_{1}}u_{n}\rightharpoonup 0\text{ in }L^{2}(\Omega )\text{,  weakly}.
\label{convergence-faible}
\end{equation}

\textbf{\ 1) we have }$u_{\infty }\in H_{0}^{1}(\Omega ;\omega _{2}):$
\ By the Mazur Lemma, there exists a sequence $(U_{n})$ of convex
combinations of $\left\{ u_{n}\right\} $ such that 
\begin{equation}
\nabla _{X_{2}}U_{n}\rightarrow \nabla _{X_{2}}u_{\infty }\text{ in }%
L^{2}(\Omega )\text{ strongly,}  \label{mazur}
\end{equation}%
then by the Lebesgue theorem there exists a subsequence $(U_{n_{k}})$ such
that: 
\begin{equation}
\text{For a.e. }X_{1}\in \omega _{1}:\nabla _{X_{2}}U_{n_{k}}(X_{1},\cdot
)\rightarrow \nabla _{X_{2}}u_{\infty }(X_{1},\cdot )\text{ in }L^{2}(\omega
_{2})\text{ strongly.}  \label{convgrad2}
\end{equation}%
Now, since $(U_{n_{k}})\in H_{0}^{1}(\Omega )^{%
\mathbb{N}
}$ then 
\begin{equation}
\text{For a.e.}X_{1}\in \omega _{1}:(U_{n_{k}}(X_{1},\cdot ))\in
H_{0}^{1}(\omega _{2})^{%
\mathbb{N}
}.  \label{Homega2}
\end{equation}
Combining (\ref{convgrad2}) and (\ref{Homega2}) we deduce: 
\begin{equation*}
\text{For a.e}.\text{ }X_{1}\in \omega _{1},\text{ }u_{\infty }(X_{1},\cdot
)\in H_{0}^{1}(\omega _{2}),
\end{equation*}%
and the proof of $u_{\infty }\in H_{0}^{1}(\Omega ;\omega _{2})$ is finished.

\textbf{2) we have }$u_{\infty }=u:$
Passing to the limit in (\ref{suite-Un}) by using (\ref{convergence-faible})
we obtain%
\begin{equation}
\int_{\Omega }\nabla _{X_{2}}u_{\infty }\cdot \nabla _{X_{2}}\varphi
dx=\int_{\Omega }\nabla _{X_{2}}u\cdot \nabla _{X_{2}}\varphi dx\text{, }%
\forall \varphi \in H_{0}^{1}(\Omega ).  \label{equation-limite}
\end{equation}%
For every $\varphi _{1}\in H_{0}^{1}(\omega _{1})$ and $\varphi _{2}\in
H_{0}^{1}(\omega _{2})$ take $\varphi =\varphi _{1}\otimes \varphi _{2}$
in (\ref{equation-limite}) we obtain, for a.e. $X_{1}\in \omega _{1}$ 
\begin{equation*}
\int_{\omega _{2}}\nabla _{X_{2}}u_{\infty }(X_{1},\cdot )\cdot \nabla
_{X_{2}}\varphi _{2}dX_{2}=\int_{\omega _{2}}\nabla _{X_{2}}u(X_{1},\cdot
)\cdot \nabla _{X_{2}}\varphi _{2}dX_{2}\text{, }\forall \varphi _{2}\in
H_{0}^{1}(\omega _{2}).
\end{equation*}%
For a.e. $X_{1}\in \omega _{1},$ take $\varphi _{2}=u_{\infty }(X_{1},\cdot
)-u(X_{1},\cdot )$ (which belongs to $H_{0}^{1}(\omega _{2})$) in the above
equality, we get:%
\begin{equation*}
\int_{\omega _{2}}\left\vert \nabla _{X_{2}}(u_{\infty }(X_{1},\cdot
)-u(X_{1},\cdot ))\right\vert ^{2}dX_{2}=0.
\end{equation*}%
Integrating over $\omega _{1}$ we deduce 
\begin{equation*}
\int_{\Omega }\left\vert \nabla _{X_{2}}(u_{\infty }-u)\right\vert ^{2}dx=0.
\end{equation*}%
Finally, since $\left\Vert \nabla _{X_{2}}(\cdot )\right\Vert
_{L^{2}(\Omega )}$ is a norm on $H_{0}^{1}(\Omega ;\omega _{2})$ we get, 
\begin{equation}
u_{\infty }=u.  \label{fin}
\end{equation}%
Combining $(\ref{mazur})$ and $(\ref{fin})$ we get the desired result.
\end{proof}

\begin{rmk}
\label{rem:density} By symmetry, $H_{0}^{1}(\Omega )$ is dense in the space 
\begin{equation*}
H_{0}^{1}(\Omega ;\omega _{1})=\left\{ u\in L^{2}(\Omega )\mid \nabla
_{X_{1}}u\in L^{2}(\Omega )\text{, and for a.e. }X_{2}\in \omega
_{2},u(\cdot ,X_{2})\in H_{0}^{1}(\omega _{1})\right\} ,
\end{equation*}%
normed by $\left\Vert \nabla _{X_{1}}(\cdot )\right\Vert _{L^{2}(\Omega )}.$
\end{rmk}

\begin{lem}
\label{lemdensitysob} The space $H_{0}^{1}(\omega _{1})\otimes
H_{0}^{1}(\omega _{2})$ is dense in $H_{0}^{1}(\Omega )$.
\end{lem}

\begin{proof}
It is well known that $D(\omega _{1})\otimes D(\omega _{2})$ is dense in $%
D(\omega _{1}\times \omega _{2})$. Here, $D(\omega
_{1}\times \omega _{2})$ is equipped with its natural topology (the
inductive limit topology). It is clear that the injection of $D(\omega
_{1}\times \omega _{2})$ in $H_{0}^{1}(\omega _{1}\times \omega _{2})$ is
continuous, thanks to the inequality 
\begin{equation*}
\forall u\in D(\Omega ):\left( \int_{\Omega }\left\vert \nabla u\right\vert
^{2}dx\right) ^{\frac{1}{2}}\leq \sqrt{N\times mes(\Omega )}\times \left(
\max_{1\leq i\leq N}\sup_{Support(u)}\left\vert \partial
_{x_{i}}u\right\vert \right).
\end{equation*}
Hence, by the density rule we obtain the density of $D(\omega
_{1})\otimes D(\omega _{2})$ in $H_{0}^{1}(\Omega )$, and the lemma follows.
\end{proof}

\begin{lem}
Let $(V_{n}^{(1)})$ and $(V_{n}^{(2)})$ be two sequences of subspaces (not
necessarily of finite dimension) of $H_{0}^{1}(\omega _{1})$ and $%
H_{0}^{1}(\omega _{2})$ respectively. If $\cup V_{n}^{(1)}$ and $\cup
V_{n}^{(2)}$are dense in $H_{0}^{1}(\omega _{1})$ and $H_{0}^{1}(\omega
_{2}) $ respectively, then $vect\left(
\bigcup\limits_{n,m}(V_{n}^{(1)}\otimes V_{m}^{(2)})\right) $ is dense in $%
H_{0}^{1}(\omega _{1})\otimes H_{0}^{1}(\omega _{2})$ for the induced
topology of $H_{0}^{1}(\Omega ).$ In particular, if $(V_{n}^{(1)})$ and $%
(V_{n}^{(2)})$ are nondecreasing then $\bigcup\limits_{n}(V_{n}^{(1)}\otimes
V_{n}^{(2)})$ is dense in $H_{0}^{1}(\omega _{1})\otimes H_{0}^{1}(\omega
_{2}).$
\end{lem}

\begin{proof}
Let us start by a useful inequality. For $u\otimes v$ in $H_{0}^{1}(\omega
_{1})\otimes H_{0}^{1}(\omega _{2})$ we have :%
\begin{eqnarray}
\left\Vert u\otimes v\right\Vert _{H_{0}^{1}(\Omega )}^{2} &=&\int_{\Omega
}\left\vert \nabla _{X_{1}}(u\otimes v)\right\vert ^{2}dx+\int_{\Omega
}\left\vert \nabla _{X_{2}}(u\otimes v)\right\vert ^{2}dx  \notag \\
&=&\left( \int_{\omega _{2}}v^{2}dX_{2}\right) \times \left( \int_{\omega
_{1}}\left\vert \nabla _{X_{1}}u\right\vert ^{2}dX_{1}\right)  \notag \\
&&+\left( \int_{\omega _{1}}u^{2}dX_{1}\right) \times \left( \int_{\omega
_{2}}\left\vert \nabla _{X_{2}}v\right\vert ^{2}dX_{2}\right)  \notag \\
&\leq &C\left\Vert u\right\Vert _{H_{0}^{1}(\omega _{1})}^{2}\times
\left\Vert v\right\Vert _{H_{0}^{1}(\omega _{2})}^{2},  \label{15}
\end{eqnarray}
where we have used Fubini's theorem and Poincar\'{e}'s inequality. Here, $%
C=C_{\omega _{1}}^{2}+C_{\omega _{2}}^{2}>0$.
Now, fix $\eta >0$ and let $\varphi \otimes \psi \in H_{0}^{1}(\omega
_{1})\otimes H_{0}^{1}(\omega _{2})$, by density of $\cup V_{n}^{(1)}$ in $%
H_{0}^{1}(\omega _{1})$ there exists $n\in 
\mathbb{N}
$ and $\varphi _{n}\in V_{n}^{(1)}$ such that:%
\begin{equation*}
\left\Vert \psi \right\Vert _{H_{0}^{1}(\omega _{2})}\times \left\Vert
\varphi _{n}-\varphi \right\Vert _{H_{0}^{1}(\omega _{1})}\leq \frac{\eta }{2%
\sqrt{C}}.
\end{equation*}
Similarly by density of $\cup V_{n}^{(2)}$ in $H_{0}^{1}(\omega _{2})$,
there exits $m\in 
\mathbb{N}
$ \ ( which depends on $n$ and $\psi )$ and $\psi _{m}\in V_{m}^{(2)}$such
that 
\begin{equation*}
\left\Vert \varphi _{n}\right\Vert _{H_{0}^{1}(\omega _{1})}\times
\left\Vert \psi _{m}-\psi \right\Vert _{H_{0}^{1}(\omega _{2})}\leq \frac{%
\eta }{2\sqrt{C}}.
\end{equation*}
Whence, by using the triangle inequality and (\ref{15}) we obtain%
\begin{equation}
\left\Vert \varphi \otimes \psi -\varphi _{n}\otimes \psi _{m}\right\Vert
_{H_{0}^{1}(\Omega )}\leq \eta.  \label{16}
\end{equation}
Now, since every element of $H_{0}^{1}(\omega _{1})\otimes H_{0}^{1}(\omega
_{2})$ could be written as $\sum\limits_{i=1}^{l}\varphi _{i}\otimes \psi
_{i}$, then by using the triangle inequality and using (\ref{16}) with $\eta$ replaced by $\frac{\eta}{l}$, one gets the desired result.
\end{proof}

\begin{cor}
\label{cordensity} $vect\left( \bigcup\limits_{n,m}(V_{n}^{(1)}\otimes
V_{m}^{(2)})\right) $ is dense in $H_{0}^{1}(\Omega ).$ in particular, if $%
(V_{n}^{(1)})$ and $(V_{n}^{(2)})$ are nondecreasing, then $%
\bigcup\limits_{n}(V_{n}^{(1)}\otimes V_{n}^{(2)})$ is dense in $%
H_{0}^{1}(\Omega ).$
\end{cor}

\section{\protect\bigskip Semigroups}

\label{AppendixB}

\begin{lem}
\label{lem-tensor}Assume $(\ref{hypA1})\text{, } (\ref{hypA2})$, $(\ref{hypAd2})$ and
let $f_{1}\in L^{2}(\omega _{1}),$ $f_{2}\in L^{2}(\omega _{2})$, then for
every $\mu >0$ we have: 
\begin{equation*}
R_{0,\mu }(f_{1}\otimes f_{2})=f_{1}\otimes (R_{0,\mu }f_{2}).
\end{equation*}%
Notice that $R_{0,\mu }f_{2}\in H_{0}^{1}(\omega _{2}).$ Moreover, we have 
\begin{equation*}
\mathcal{A}_{0,\mu }(f_{1}\otimes f_{2})=f_{1}\otimes (\mathcal{A}_{0,\mu
}f_{2}).
\end{equation*}
Notice also that $\mathcal{A}_{0,\mu }f_{2}\in L^{2}(\omega _{2}).$ Here, $%
\mathcal{A}_{0,\mu }$ is the Yosida approximation of $\mathcal{A}_{0}$, recall
that $\mathcal{A}_{0,\mu }=\mu \mathcal{A}_{0}R_{0,\mu }$.
\end{lem}

\begin{proof}
Let $v_{2}\in H_{0}^{1}(\omega _{2})$ be the unique solution in $%
H_{0}^{1}(\omega _{2})$ to 
\begin{equation}
\mu \int_{\omega _{2}}v_{2}\varphi _{2}dX_{2}+\int_{\omega
_{2}}A_{22}(X_{2})\nabla _{X_{2}}v_{2}\cdot \nabla _{X_{2}}\varphi
_{2}dX_{2}=\int_{\omega _{2}}f_{2}\text{ }\varphi _{2}dX_{2}\text{, }\forall
\varphi _{2}\in H_{0}^{1}(\omega _{2}),  \label{1}
\end{equation}%
Let $\varphi \in H_{0}^{1}(\Omega ;\omega _{2}),$ then $\varphi (X_{1},\cdot
)\in H_{0}^{1}(\omega _{2})$ for a.e. $X_{1}\in\omega _{1}$. Let $f_{1}\in L^{2}(\omega _{1}),$ multiplying
(\ref{1}) by $f_{1}$, testing in (\ref{1}) with $\varphi (X_{1},\cdot )$ and integrating over $\omega _{1}$ yields 
\begin{equation*}
\mu \int_{\Omega }f_{1}v_{2}\varphi dx+\int_{\Omega }A_{22}(X_{2})\nabla
_{X_{2}}(f_{1}v_{2})\cdot \nabla _{X_{2}}\varphi dx=\int_{\Omega }f_{1}f_{2}%
\text{ }\varphi dx.
\end{equation*}%
It is clear that $f_{1}v_{2}\in H_{0}^{1}(\Omega ;\omega _{2})$ whence, $%
R_{0,\mu }(f_{1}\otimes f_{2})=f_{1}\otimes v_{2}$, in particular when $%
f_{1}=1$ we have $R_{0,\mu }(f_{2})=v_{2},$ and therefore $R_{0,\mu
}(f_{1}\otimes f_{2})=f_{1}\otimes R_{0,\mu }(f_{2}).$
The second assertion follows immediately from the first one, in fact 
\begin{equation*}
\mathcal{A}_{0,\mu }(f_{1}\otimes f_{2})=\mu \mathcal{A}_{0}R_{0,\mu
}(f_{1}\otimes f_{2})=\mu \mathcal{A}_{0}(f_{1}\otimes R_{0,\mu }f_{2}).
\end{equation*}%
We have $R_{0,\mu }f_{2}\in D(\mathcal{A}_{0})\cap H_{0}^{1}(\omega _{2})$
then by using (\ref{hypAd2}) we get
\begin{equation*}
\mathcal{A}_{0}(f_{1}\otimes R_{0,\mu }f_{2})=f_{1}\otimes \mathcal{A}%
_{0}(R_{0,\mu }f_{2}),
\end{equation*}%
Notice that the operator $\mathcal{A}_{0}$ is independent of 
the $X_{1}$ direction and that $\mathcal{A}_{0}(R_{0,\mu }f_{2})\in L^{2}(\omega _{2})$.
Finally we get 
\begin{equation*}
\mathcal{A}_{0,\mu }(f_{1}\otimes f_{2})=\mu f_{1}\otimes \mathcal{A}%
_{0}(R_{0,\mu }f_{2})=f_{1}\otimes \mathcal{A}_{0,\mu }(f_{2}).
\end{equation*}
\end{proof}

Now, let $s\geq 0$, $\mu >0$ and $g\in L^{2}(\Omega )$. To simplify the notations, we denote $%
f_{g}:=e^{s\mathcal{A}_{0,\mu }}g$  instead of $f_{g,s,\mu}$. 

\begin{lem}
\label{lem-tensor-semigroup}Assume $(\ref{hypA1})$, $(\ref{hypA2}),$ $(\ref%
{hypAd2}).$ Let $g=g_{1}\otimes g_{2}\in L^{2}(\omega _{1})\otimes
L^{2}(\omega _{2})$, then for $s\geq 0$, $\mu >0$ we have:%
\begin{equation*}
f_{g}=g_{1}\otimes e^{s\mathcal{A}_{0,\mu }}g_{2}.
\end{equation*}%
Notice that $e^{s\mathcal{A}_{0,\mu }}g_{2}\in L^{2}(\omega _{2}).$
\end{lem}

\begin{proof}
we have 
\begin{equation*}
f_{g}=e^{s\mathcal{A}_{0,\mu }}g=\sum_{k=0}^{\infty }\frac{s^{k}}{k!}%
\mathcal{A}_{0,\mu }^{k}g,
\end{equation*}%
where the series converges in $L^{2}(\Omega )$. By an immediate induction we get by using Lemma \ref{lem-tensor} 
\begin{equation*}
\forall k\in 
\mathbb{N}
:\mathcal{A}_{0,\mu }^{k}g=g_{1}\otimes \mathcal{A}_{0,\mu }^{k}g_{2},
\end{equation*}
with $\mathcal{A}_{0,\mu }^{k}g_{2}\in L^{2}(\omega _{2})$ for every $k\in 
\mathbb{N}
,$ and the Lemma follows.
\end{proof}

\begin{lem}
\label{lem-semigroup-regularity1}Assume $(\ref{hypA1})$, $(\ref{hypA2}),$ $(\ref%
{hypAd2})$. Let $g\in H^{2}(\omega _{1})\otimes
L^{2}(\omega _{2})$, then for $s\geq0$, $\mu >0,$ $i,j=1,...,q$ we have $%
D_{x_{i}x_{j}}^{2}f_{g},$ $D_{x_{i}}f_{g}\in L^{2}(\Omega ),$ with:
\begin{equation}
D_{x_{i}x_{j}}^{2}f_{g}=e^{s\mathcal{A}_{0,\mu }}(D_{x_{i}x_{j}}^{2}g),\text{
}D_{x_{i}}f_{g}=e^{s\mathcal{A}_{0,\mu }}(D_{x_{i}}g).  \label{ass1}
\end{equation}%
\begin{equation}
\left\Vert D_{x_{i}x_{j}}^{2}f_{g}\right\Vert _{L^{2}(\Omega )}\leq
\left\Vert D_{x_{i}x_{j}}^{2}g\right\Vert _{L^{2}(\Omega )},\text{ }%
\left\Vert D_{x_{i}}f_{g}\right\Vert _{L^{2}(\Omega )}\leq \left\Vert
D_{x_{i}}g\right\Vert _{L^{2}(\Omega )}.  \label{ass2}
\end{equation}
\end{lem}

\begin{proof}
1) Suppose the simple case when $g=g_{1}\otimes g_{2}.$
So, let $g=g_{1}\otimes g_{2}\in H^{2}(\omega _{1})\otimes L^{2}(\omega _{2})
$ and let us prove assertions (\ref{ass1}). By Lemma \ref{lem-tensor-semigroup} we get 
\begin{equation*}
f_{g}=g_{1}\otimes e^{s\mathcal{A}_{0,\mu }}(g_{2})\text{, }
\end{equation*}%
with $e^{s\mathcal{A}_{0,\mu }}g_{2}\in L^{2}(\omega _{2}).$Hence, for $%
i,j=1,...,q$ we get $D_{x_{i}x_{j}}^{2}f_{g}\in L^{2}(\Omega )$ and $%
D_{x_{i}x_{j}}^{2}f_{g}=\left( D_{x_{i}x_{j}}^{2}g_{1}\right) \otimes e^{s%
\mathcal{A}_{0,\mu }}g_{2}$, and applying Lemma \ref{lem-tensor-semigroup} we get%
\begin{equation*}
D_{x_{i}x_{j}}^{2}f_{g}=e^{s\mathcal{A}_{0,\mu }}(D_{x_{i}x_{j}}^{2}g).
\end{equation*}
Similarly we get $D_{x_{i}}f_{g}=e^{s\mathcal{A}_{0,\mu }}(D_{x_{i}}g)$, and assertion (\ref{ass1}) follows when $g=g_{1}\otimes g_{2}.$

2) Now, let $g\in H^{2}(\omega _{1})\otimes L^{2}(\omega _{2}),$ since $g$
is a finite sum of elements of the form $g_{1}\otimes g_{2}$, then by
linearity we get $D_{x_{i}x_{j}}^{2}f_{g}, \text{ }D_{x_{i}}f_{g}\in L^{2}(\Omega )$ and 
\begin{equation*}
D_{x_{i}x_{j}}^{2}f_{g}=e^{s\mathcal{A}_{0,\mu }}(D_{x_{i}x_{j}}^{2}g)\text{%
, }D_{x_{i}}f_{g}=e^{s\mathcal{A}_{0,\mu }}(D_{x_{i}}g), \text{ for }%
i,j=1,...,q,
\end{equation*}%
therefore%
\begin{equation*}
\left\Vert D_{x_{i}x_{j}}^{2}f_{g}\right\Vert _{L^{2}(\Omega )}\leq
\left\Vert e^{s\mathcal{A}_{0,\mu }}\right\Vert \left\Vert
D_{x_{i}x_{j}}^{2}g\right\Vert _{L^{2}(\Omega )}\leq \left\Vert
D_{x_{i}x_{j}}^{2}g\right\Vert _{L^{2}(\Omega )}, \text{ for } %
i,j=1,...,q,
\end{equation*}%
and similarly we obtain the second inequality of $(\ref{ass2})$.
\end{proof}

\begin{lem}
\label{lem-semigroup-regularity2}Assume $(\ref{hypA1})$, $(\ref{hypA2}),$ $(\ref%
{hypAd2})$ and $(\ref{A-liptschitz}).$ Let $g\in \left( H_{0}^{1}\cap
H^{2}(\omega _{1})\right) \otimes (H_{0}^{1}\cap H^{2}(\omega _{2}))$, then
for $s\geq 0$, $\mu >0$ we have: 
\begin{equation}
f_{g}\in D(\mathcal{A}_{0})\text{, } \mathcal{A}_{0}(f_{g})\in H_{0}^{1}(\Omega ;\omega _{1})\text{, and }%
D_{x_{i}}\left( \mathcal{A}_{0}f_{g}\right) =e^{s\mathcal{A}_{0,\mu
}}(D_{x_{i}}\mathcal{A}_{0}g),\text{ }i=1,...,q,  \label{ass3}
\end{equation}
\begin{equation}
\left\Vert\left( \mathcal{A}_{0}f_{g}\right) \right\Vert
_{L^{2}(\Omega )}\leq \left\Vert \mathcal{A}_{0}g\right\Vert
_{L^{2}(\Omega )} \text{and \ } 
\left\Vert D_{x_{i}}\left( \mathcal{A}_{0}f_{g}\right) \right\Vert
_{L^{2}(\Omega )}\leq \left\Vert D_{x_{i}}\mathcal{A}_{0}g\right\Vert
_{L^{2}(\Omega )},\text{ }i=1,...,q.  \label{ass4}
\end{equation}
\end{lem}

\begin{proof}
1) Suppose $g=g_{1}\otimes g_{2}\in (H_{0}^{1}\cap H^{2}(\omega
_{1}))\otimes (H_{0}^{1}\cap H^{2}(\omega _{2}))$ and let us prove $(\ref%
{ass3})$. Since $g\in D(\mathcal{A}_{0})$, thanks to (\ref{A-liptschitz}), then $f_{g}=e^{s%
\mathcal{A}_{0,\mu }}g\in D(\mathcal{A}_{0})$\ and $\mathcal{A}_{0}f_{g}=e^{s\mathcal{A}_{0,\mu }}\mathcal{A}_{0}g$ ( thanks to (\ref{abstact-commute})).
Now, we have 
\begin{equation*}
\mathcal{A}_{0}f_{g}=\mathcal{A}_{0}(e^{s\mathcal{A}_{0,\mu }}g)=\mathcal{A}%
_{0}\left( g_{1}\otimes e^{s\mathcal{A}_{0,\mu }}g_{2}\right) .
\end{equation*}
Notice that, $g_{2}\in D(\mathcal{A}_{0})$, thanks to (\ref{A-liptschitz}), then $e^{s\mathcal{A}_{0,\mu
}}g_{2}\in D(\mathcal{A}_{0})$ ( thanks to (\ref{abstact-commute}))$,$ hence 
\begin{equation*}
\mathcal{A}_{0}f_{g}=g_{1}\mathcal{A}_{0}e^{s\mathcal{A}_{0,\mu }}g_{2},
\end{equation*}%
where we have used the fact that $\mathcal{A}_{0}$ is
independent of the $X_{1}-direction.$ Since $e^{s\mathcal{A}%
_{0,\mu }}$ and $\mathcal{A}_{0}$ \ commute on $D(\mathcal{A}_{0}),$ then 
\begin{equation*}
\mathcal{A}_{0}f_{g}=g_{1}e^{s\mathcal{A}_{0,\mu }}\mathcal{A}_{0}g_{2}.
\end{equation*}%
Now, we have $\mathcal{A}_{0}g_{2}\in L^{2}(\omega _{2})$ then $e^{s\mathcal{%
A}_{0,\mu }}\mathcal{A}_{0}g_{2}\in L^{2}(\omega _{2})$ (thanks to Lemma %
\ref{lem-tensor-semigroup}), however $g_{1}\in H_{0}^{1}(\omega _{1})$, then $\mathcal{A}_{0}f_{g}\in H_{0}^{1}(\Omega ;\omega _{1})$. Whence, for $i=1,...,q$ we have 
\begin{equation*}
D_{x_{i}}\left( \mathcal{A}_{0}f_{g}\right) =D_{x_{i}}g_{1}\otimes e^{s%
\mathcal{A}_{0,\mu }}\mathcal{A}_{0}g_{2},
\end{equation*}%
and hence by, Lemma \ref{lem-tensor-semigroup} we get 
\begin{eqnarray*}
D_{x_{i}}\left( \mathcal{A}_{0}f_{g}\right)  &=&e^{s\mathcal{A}_{0,\mu
}}\left( D_{x_{i}}g_{1}\otimes \mathcal{A}_{0}g_{2}\right)  \\
&=&e^{s\mathcal{A}_{0,\mu }}\left( D_{x_{i}}\mathcal{A}_{0}g\right) .
\end{eqnarray*}
(Remark that $D_{x_{i}}\mathcal{A}_{0}g\in L^{2}(\Omega )$ since $g_{1}\in
H_{0}^{1}(\omega _{1})$ and $\mathcal{A}_{0}g_{2}\in L^{2}(\omega _{2})).$ \\
2) Now, for a general $g\in (H_{0}^{1}\cap H^{2}(\omega _{1}))\otimes
(H_{0}^{1}\cap H^{2}(\omega _{2})),$ assertion (\ref{ass3}) follows by linearity.
Finally, we show $(\ref{ass4})$. We have 
\begin{eqnarray*}
\left\Vert\left( \mathcal{A}_{0}f_{g}\right) \right\Vert
_{L^{2}(\Omega )} &=&\left\Vert e^{s\mathcal{A}_{0,\mu }}(\mathcal{A%
}_{0}g)\right\Vert _{L^{2}(\Omega )}\leq \left\Vert e^{s\mathcal{A}_{0,\mu
}}\right\Vert \left\Vert \mathcal{A}_{0}g\right\Vert _{L^{2}(\Omega
)} \\
&\leq &\left\Vert\mathcal{A}_{0}g\right\Vert _{L^{2}(\Omega )}.
\end{eqnarray*}
For $i=1,...,q$ we get%
\begin{eqnarray*}
\left\Vert D_{x_{i}}\left( \mathcal{A}_{0}f_{g}\right) \right\Vert
_{L^{2}(\Omega )} &=&\left\Vert e^{s\mathcal{A}_{0,\mu }}(D_{x_{i}}\mathcal{A%
}_{0}g)\right\Vert _{L^{2}(\Omega )}\leq \left\Vert e^{s\mathcal{A}_{0,\mu
}}\right\Vert \left\Vert D_{x_{i}}\mathcal{A}_{0}g\right\Vert _{L^{2}(\Omega
)} \\
&\leq &\left\Vert D_{x_{i}}\mathcal{A}_{0}g\right\Vert _{L^{2}(\Omega )}.
\end{eqnarray*}
\end{proof}

\begin{lem}
\label{lem-semigroup-regularity3}Assume $(\ref{hypA1})$, $(\ref{hypA2}),$ $(\ref%
{hypAd2})$ and $(\ref{A-liptschitz}).$ Let $g\in \left( H_{0}^{1}\cap
H^{2}(\omega _{1})\right) \otimes (H_{0}^{1}\cap H^{2}(\omega _{2}))$, then
for $s\geq 0$, $\mu >0$, $i=1,...,q$, $j=q+1,...,N$ we have $%
D_{x_{j}}f_{g},D_{x_{i}x_{j}}^{2}f_{g}\in L^{2}(\Omega )$ with: 
\begin{equation}
\left\Vert D_{x_{j}}f_{g}\right\Vert _{L^{2}(\Omega )}^{2}\leq \frac{1}{%
\lambda}\left\Vert \mathcal{A}_{0}g\right\Vert _{L^{2}(\Omega
)}\left\Vert g\right\Vert _{L^{2}(\Omega )},\text{ }\left\Vert
D_{x_{i}x_{j}}^{2}f_{g}\right\Vert _{L^{2}(\Omega )}^{2}\leq \frac{1}{%
\lambda}\left\Vert D_{x_{i}}\mathcal{A}_{0}g\right\Vert _{L^{2}(\Omega
)}\left\Vert D_{x_{i}}g\right\Vert _{L^{2}(\Omega )}.  \label{ass5}
\end{equation}
\end{lem}

\begin{proof}
1) Let us show the first inequality of(\ref{ass5}). Suppose $g\in (H_{0}^{1}\cap H^{2}(\omega
_{1}))\otimes (H_{0}^{1}\cap H^{2}(\omega _{2}))$. Notice that $g\in D(%
\mathcal{A}_{0})$, thanks to (\ref{A-liptschitz}), then according to (\ref{abstact-commute}) we have $%
f_{g}\in D(\mathcal{A}_{0})\subset H_{0}^{1}(\Omega ;\omega _{2})$, hence for 
$j\in \{q+1,...,N\}$  the ellipticity assumption gives
\begin{eqnarray*}
\left\Vert D_{x_{j}}f_{g}\right\Vert _{L^{2}(\Omega )}^{2} &\leq &\frac{1}{%
\lambda}\left\langle -\mathcal{A}_{0}f_{g},f_{g}\right\rangle
_{L^{2}(\Omega )} \\
&\leq &\frac{1}{\lambda}\left\Vert \mathcal{A}_{0}f_{g}\right\Vert
_{L^{2}(\Omega )}\left\Vert f_{g}\right\Vert _{L^{2}(\Omega )}.
\end{eqnarray*}%
We have, $\left\Vert \mathcal{A}_{0}f_{g}\right\Vert _{L^{2}(\Omega
)}=\left\Vert \mathcal{A}_{0}e^{s\mathcal{A}_{0,\mu }}g\right\Vert
_{L^{2}(\Omega )}=\left\Vert e^{s\mathcal{A}_{0,\mu }}\mathcal{A}%
_{0}g\right\Vert _{L^{2}(\Omega )}\leq \left\Vert \mathcal{A}%
_{0}g\right\Vert _{L^{2}(\Omega )}$, and $\left\Vert f_{g}\right\Vert
_{L^{2}(\Omega )}\leq \left\Vert g\right\Vert _{L^{2}(\Omega )}$, therefore 
\begin{equation*}
\left\Vert D_{x_{j}}f_{g}\right\Vert _{L^{2}(\Omega )}^{2}\leq \frac{1}{%
\lambda}\left\Vert \mathcal{A}_{0}g\right\Vert _{L^{2}(\Omega
)}\left\Vert g\right\Vert _{L^{2}(\Omega )}.
\end{equation*}
2) Now, let $1\leq i\leq q$ fixed, then according to Lemma \ref%
{lem-semigroup-regularity1} we have $D_{x_{i}}f_{g}=e^{s\mathcal{A}_{0,\mu}}(D_{x_{i}}g),$ notice that $D_{x_{i}}g\in D(\mathcal{A}_{0})$ and hence, $D_{x_{i}}f_{g}\in D(\mathcal{A}_{0})$, in
particular $D_{x_{i}}f_{g}\in H_{0}^{1}(\Omega ;\omega _{2})$, and for $%
q+1\leq j\leq N$ we get 
\begin{eqnarray*}
\left\Vert D_{x_{j}x_{i}}^{2}f_{g}\right\Vert _{L^{2}(\Omega )}^{2} &\leq &%
\frac{1}{\lambda}\left\langle -\mathcal{A}%
_{0}D_{x_{i}}f_{g},D_{x_{i}}f_{g}\right\rangle _{L^{2}(\Omega )} \\
&\leq &\frac{1}{\lambda}\left\Vert \mathcal{A}_{0}D_{x_{i}}f_{g}\right%
\Vert _{L^{2}(\Omega )}\left\Vert D_{x_{i}}f_{g}\right\Vert _{L^{2}(\Omega
)}.
\end{eqnarray*}%
We have, 
\begin{eqnarray*}
\left\Vert \mathcal{A}_{0}D_{x_{i}}f_{g}\right\Vert _{L^{2}(\Omega )}
&=&\left\Vert \mathcal{A}_{0}e^{s\mathcal{A}_{0,\mu
}}(D_{x_{i}}g)\right\Vert _{L^{2}(\Omega )}=\left\Vert e^{s\mathcal{A}%
_{0,\mu }}(\mathcal{A}_{0}D_{x_{i}}g)\right\Vert _{L^{2}(\Omega )} \\
&\leq &\left\Vert (D_{x_{i}}\mathcal{A}_{0}g)\right\Vert _{L^{2}(\Omega )}.
\end{eqnarray*}%
Finally, by using (\ref{ass2}) and the above inequality we obtain
\begin{equation*}
\left\Vert D_{x_{j}x_{i}}^{2}f_{g}\right\Vert _{L^{2}(\Omega )}^{2}\leq 
\frac{1}{\lambda}\left\Vert (D_{x_{i}}\mathcal{A}_{0}g)\right\Vert
_{L^{2}(\Omega )}\left\Vert D_{x_{i}}g\right\Vert _{L^{2}(\Omega )}.
\end{equation*}
\end{proof}
\begin{lem}
 \label{lem-semigroup-regularity4}Under assumptions of Lemma $\ref%
{lem-semigroup-regularity3}$, we have for $g\in \left( H_{0}^{1}\cap
H^{2}(\omega _{1})\right) \otimes (H_{0}^{1}\cap H^{2}(\omega _{2})):$ 
\begin{equation}
f_{g}\in H_{0}^{1}(\Omega )\cap D(\mathcal{A}_{0}),  \label{ass6}
\end{equation}%
and 
\begin{equation}
\text{div}_{X_{1}}(A_{11}\nabla _{X_{1}}f),\text{ div}_{X_{1}}(A_{12}\nabla
_{X_{2}}f),\text{ div}_{X_{2}}(A_{21}\nabla _{X_{1}}f)\in L^{2}(\Omega).
\label{ass7}
\end{equation}
\end{lem}

\begin{proof}
Let us prove (\ref{ass6}). In Lemma \ref{lem-semigroup-regularity2} we proved that $f_{g}\in D(\mathcal{A}_{0}).$ Let us show that $f_{g}\in
H_{0}^{1}(\Omega ).$ Suppose the simple case $g=g_{1}\otimes g_{2},$ we have $%
f_{g}=g_{1}\otimes e^{s\mathcal{A}_{0,\mu }}g_{2}$. Since $g_{2}\in D(%
\mathcal{A}_{0}),$ then $e^{s\mathcal{A}_{0,\mu }}g_{2}\in D(\mathcal{A}%
_{0}),$ in particular we have $e^{s\mathcal{A}_{0,\mu }}g_{2}\in
H_{0}^{1}(\Omega ;\omega _{2})$ however, according to Lemma \ref%
{lem-tensor-semigroup} $e^{s\mathcal{A}_{0,\mu }}g_{2}\in L^{2}(\omega _{2}),
$ hence $e^{s\mathcal{A}_{0,\mu }}g_{2}\in H_{0}^{1}(\omega _{2}).$ Finally
as $g_{1}\in H_{0}^{1}(\omega _{1})$ we get $f_{g}\in H_{0}^{1}(\Omega ).$
For a general $g$ in the tensor product space, the proof follows by
linearity. 

Now, let us show (\ref{ass7}). According to Lemmas \ref%
{lem-semigroup-regularity1}, \ref{lem-semigroup-regularity3} all these
derivatives $D_{x_{i}}f_{g},$ $D_{x_{i}x_{j}}^{2}f_{g}$ for $1\leq i,j\leq q$%
, and $D_{x_{j}}f_{g},$ $D_{x_{i}x_{j}}^{2}f_{g}$ for $1\leq i\leq q$, $%
q+1\leq j\leq N$ belong to $L^{2}(\Omega ).$ Whence, combining this with (%
\ref{A-liptschitz}) we get (\ref{ass7}).
\end{proof}

\section{Existence theorem}

\label{AppendixC}

Let $V\subset H_{0}^{1}(\Omega )$ be a subspace. We consider the problem 
\begin{equation}
\left\{ 
\begin{array}{l}
\int_{\Omega }\beta (u)\varphi dx+\int_{\Omega }A_{22}\nabla _{X_{2}}u\cdot
\nabla _{X_{2}}\varphi dx=\int_{\Omega }f\text{ }\varphi dx\text{, }%
\forall \varphi \in V\text{\ \ \ } \\ 
u\in V\text{, }%
\end{array}%
\right.  \label{schauderlimit}
\end{equation}%
with $A_{22}$ and $\beta $ as in the introduction.

\begin{prop}
If $V$ is closed in $H_{0}^{1}(\Omega ;\omega _{2})$, then there exists a
solution to $(\ref{schauderlimit}).$
\end{prop}

\begin{proof}
We consider the perturbed problem 
\begin{equation}
\left\{ 
\begin{array}{l}
\int_{\Omega }\beta (u_{\epsilon})\varphi dx+\int_{\Omega }\tilde{A}_{\epsilon }\nabla u_{\epsilon}\cdot
\nabla\varphi dx=\int_{\Omega }f\text{ }\varphi dx\text{, }%
\forall \varphi \in V\text{\ \ \ } \\ 
u_{\epsilon}\in V\text{, }%
\end{array}%
\right.  \label{schauder-pertub}
\end{equation}%
with 
\begin{equation*}
\tilde{A}_{\epsilon }=%
\begin{pmatrix}
\epsilon ^{2}I_{q} && 0 \\
0 && A_{22}
\end{pmatrix}.
\end{equation*}%
The space $V$ is closed in $H_{0}^{1}(\Omega )$, thanks to the continuous embedding $H_{0}^{1}(\Omega )\hookrightarrow H_{0}^{1}(\Omega; \omega_{2} ) $. The function $\tilde{A}_{\epsilon }$
is bounded and coercive, then by using the Schauder fixed point theorem, one can show  the existence of a solution $u_{\epsilon }$ to (\ref{schauder-pertub}). This solution
is unique in $V$ thanks to (\ref{hypBeta1}) and coercivity of $\tilde{A}%
_{\epsilon }.$ Testing with $u_{\epsilon }$ in (\ref{schauder-pertub}) we obtain 
\begin{equation*}
\epsilon \left\Vert \nabla _{X_{1}}u_{\epsilon }\right\Vert _{L^{2}(\Omega )}%
\text{, }\left\Vert \nabla _{X_{2}}u_{\epsilon }\right\Vert _{L^{2}(\Omega )}%
\text{, }\left\Vert u_{\epsilon }\right\Vert _{L^{2}(\Omega )}\leq C\text{,}
\end{equation*}
where $C$ is independent of $\epsilon $, we have used that $\int_{\Omega
}\beta (u_{\epsilon })u_{\epsilon }dx\geq 0$ (thanks to (\ref{hypBeta1})). By using (\ref{hypBeta2}), we get
\begin{equation*}
\left\Vert \beta (u_{\epsilon })\right\Vert _{L^{2}(\Omega )}\leq
M(\left\vert \Omega \right\vert ^{\frac{1}{2}}+C).
\end{equation*}
So, there exist $v\in L^{2}(\Omega )$, $u\in L^{2}(\Omega )$ with $\nabla
_{X_{2}}u\in L^{2}(\Omega )$, and a subsequence $(u_{\epsilon _{k}})_{k\in 
\mathbb{N}
}$ such that 
\begin{equation}
\beta (u_{\epsilon _{k}})\rightharpoonup v\text{, }\epsilon _{k}\nabla
_{X_{1}}u_{\epsilon _{k}}\rightharpoonup 0\text{, }\nabla
_{X_{2}}u_{\epsilon _{k}}\rightharpoonup \nabla _{X_{2}}u\text{, }%
u_{\epsilon _{k}}\rightharpoonup u\text{ in }L^{2}(\Omega )\text{-weak}.
\label{conv-faible}
\end{equation}
Passing to the limit in (\ref{schauder-pertub}) we get%
\begin{equation}
\int_{\Omega }v\varphi dx+\int_{\Omega }A_{22}\nabla _{X_{2}}u\cdot \nabla
_{X_{2}}\varphi dx=\int_{\Omega }f\varphi dx\ \text{,}\ \forall \varphi \in V.%
\text{\ }  \label{conv-faible-limit-bis}
\end{equation}
Take $\varphi =u_{\epsilon _{k}}$ in (\ref{conv-faible-limit-bis}) and passing to
the limit we get%
\begin{equation}
\int_{\Omega }vudx+\int_{\Omega }A_{22}\nabla _{X_{2}}u\cdot \nabla
_{X_{2}}udx=\int_{\Omega }fudx\   \label{conv-faible-limit}
\end{equation}
Let us consider the quantity 
\begin{multline*}
0\leq I_{k}=\int_{\Omega }\epsilon ^{2}\left\vert \nabla _{X_{1}}u_{\epsilon
_{k}}\right\vert ^{2}dx+\int_{\Omega }A_{22}\nabla _{X_{2}}(u_{\epsilon
_{k}}-u)\cdot \nabla _{X_{2}}(u_{\epsilon _{k}}-u) \\
+\int_{\Omega }(\beta (u_{\epsilon _{k}})-\beta (u))(u_{\epsilon _{k}}-u)dx
\\
=\int_{\Omega }fu_{\epsilon _{k}}dx-\int_{\Omega }A_{22}\nabla
_{X_{2}}u_{\epsilon _{k}}\cdot \nabla _{X_{2}}udx-\int_{\Omega }A_{22}\nabla
_{X_{2}}u\cdot \nabla _{X_{2}}u_{\epsilon _{k}}dx \\
+\int_{\Omega }fudx-\int_{\Omega }vudx-\int_{\Omega }\beta (u)u_{\epsilon
_{k}}dx \\
-\int_{\Omega }\beta (u_{\epsilon _{k}})udx+\int_{\Omega }\beta (u)udx
\end{multline*}
Remark that this quantity is nonnegative, thanks to the ellipticity and monotonicity assumptions.
Passing to the limit as $k\rightarrow \infty $ using (\ref{conv-faible}), (%
\ref{conv-faible-limit}) we get 
\begin{equation*}
\lim I_{k}=0.
\end{equation*}
Therefore, the ellipticity assumption shows that 
\begin{equation}
\left\Vert \epsilon _{k}\nabla _{X_{1}}u_{\epsilon _{k}}\right\Vert
_{L^{2}(\Omega )}\text{,}\left\Vert u_{\epsilon _{k}}-u\right\Vert
_{L^{2}(\Omega )},\text{ }\left\Vert \nabla _{X_{2}}(u_{\epsilon
_{k}}-u)\right\Vert _{L^{2}(\Omega )}\rightarrow 0,
\label{convergence forte-limit}
\end{equation}%
and hence, by a contradiction argument one has
\begin{equation*}
\beta (u_{\epsilon _{k}})\rightarrow \beta (u)\text{ strongly in }L^{2}(\Omega).
\end{equation*}
Whence (\ref{conv-faible-limit-bis}) becomes%
\begin{equation*}
\int_{\Omega }\beta (u)\varphi dx+\int_{\Omega }A_{22}\nabla _{X_{2}}u\cdot
\nabla _{X_{2}}\varphi dx=\int_{\Omega }f\varphi dx\ \text{,}\ \forall
\varphi \in V.
\end{equation*}
Finally, 
$\left\Vert \nabla _{X_{2}}(u_{\epsilon _{k}}-u)\right\Vert _{L^{2}(\Omega
)}\rightarrow 0$ shows that $u\in H_{0}^{1}(\Omega ;\omega _{2})$, and
therefore as $V$ is closed in $H_{0}^{1}(\Omega ;\omega _{2})$ we obtain that $u\in
V$.
\end{proof}

\end{document}